\documentclass[a4paper, 11pt]{amsproc}

\usepackage{amsmath, amsthm, amssymb, mathtools, mathrsfs, stmaryrd}
\usepackage{bm}
\usepackage{shuffle}

\usepackage{hyperref}

\usepackage{xcolor}
\usepackage[capitalize,nameinlink,noabbrev,nosort]{cleveref}
\hypersetup{
	colorlinks=true,       
	linkcolor=blue,          
	citecolor=blue,        
	filecolor=blue,      
	urlcolor=blue,           
}

\usepackage{fullpage}

\makeatletter
\@namedef{subjclassname@2020}{%
  \textup{2020} Mathematics Subject Classification}
\makeatother

\newtheorem{theoremcounter}{Theorem Counter}[section]

\theoremstyle{definition}

\newtheorem{remark}[theoremcounter]{Remark}

\theoremstyle{plain}
\newtheorem{lemma}[theoremcounter]{Lemma}
\newtheorem{proposition}[theoremcounter]{Proposition}
\newtheorem{corollary}[theoremcounter]{Corollary}

\newtheorem{theorem}[theoremcounter]{Theorem}

\numberwithin{equation}{section}

\newcommand{\R}{\mathbb{R}}
\newcommand{\C}{\mathbb{C}}

\DeclareMathOperator{\ImNew}{Im}
\renewcommand{\Im}{\ImNew}
\DeclareMathOperator{\ReNew}{Re}
\renewcommand{\Re}{\ReNew}

\begin{document}

\title{Lower-order terms in mean-square formulas for the Euler--Zagier double zeta-function and regularized multiple zeta values}


\author{Tomokazu Onozuka}
\address[Tomokazu Onozuka]{Faculty of Science and Technology, Oita University, 700 Dannoharu, Oita, 870-1192, Japan} \email{t-onozuka@oita-u.ac.jp}

\begin{abstract}
We study the mean square of the Euler--Zagier double zeta-function when the second variable moves vertically. 
Mean-square formulas in and beyond the region of absolute convergence were previously obtained by Matsumoto and Tsumura and by Ikeda, Matsuoka and Nagata. 
In particular, the latter authors determined the leading terms in the following three boundary cases:
\[
 \zeta_2\left(1+a,\frac12+it\right),\qquad
 \zeta_2\left(1+ib,\frac12+it\right),\qquad
 \zeta_2\left(1-a,\frac12+a+it\right),
\]
where $a\in\mathbb C$ with $\Re a>0$ and $b\in\mathbb R$ are fixed. 
The mean squares in the first and third cases, as well as in the second case when $b\neq0$, have leading terms of order $T\log T$. 
When $b=0$, the second case reduces to the corner $(1,1/2)$, where the leading term is of order $T(\log T)^3$.

In the present paper, we refine these results by obtaining asymptotic formulas with remainder $o(T)$. 
In the case $b\neq0$, an additional oscillatory term of order $T$ occurs.
We also show that, at the parameter points $(s_1,\sigma_2)=(k,\ell/2)$ with integers $k\geq1$ and
$\ell\geq2$, the constant multiplying $T$ in the mean-square formula is a linear combination of multiple zeta values.
At the boundary points $(k,1/2)$ with integers $k\geq1$, the harmonic finite parts of the corresponding divergent sums are expressed in terms of harmonic regularized multiple zeta values.
\end{abstract}

\maketitle

\section{Introduction}
\label{sec:introduction}

The mean square of the Riemann zeta-function is one of the most fundamental objects in analytic number theory. 
For $\sigma>1/2$, one has
\begin{align*}
 \int_2^T|\zeta(\sigma+it)|^2dt  &\sim  \zeta(2\sigma)T,
\end{align*}
whereas on the critical line,
\begin{align}
 \int_2^T \left|\zeta\left(\frac12+it\right)\right|^2dt
 &= T\log\frac{T}{2\pi} +(2\gamma-1)T +o(T). 
 \label{eq:intro-zeta-critical}
\end{align}
Thus the form of the main term changes when $\sigma$ reaches $1/2$. 
From the viewpoint of Dirichlet series, this transition reflects the divergence of the coefficient $\zeta(2\sigma)$ at $\sigma=1/2$.

The purpose of the present paper is to study an analogous transition for the Euler--Zagier double zeta-function
\begin{align}
 \zeta_2(s_1,s_2) &:= \sum_{1\leq m<n} \frac1{m^{s_1}n^{s_2}}. 
 \label{eq:intro-double-zeta}
\end{align}
The series in \eqref{eq:intro-double-zeta} converges absolutely in the region $\{(s_1,s_2)\mid\Re s_2>1,~ \Re s_1+\Re s_2>2\}$, and $\zeta_2(s_1,s_2)$ can be continued meromorphically to $\mathbb C^2$;
see, among others, \cite{AkiyamaEgamiTanigawa,Zhao}.
The double zeta-function was already used by Atkinson \cite{Atkinson} in his study of the error term in the mean square of the Riemann zeta-function, and its analytic properties have subsequently been investigated from several points of view.

Matsumoto and Tsumura \cite{MatsumotoTsumura} introduced and studied the mean square
\[
 \int_2^T|\zeta_2(s_1,\sigma_2+it)|^2dt
\]
with $s_1$ fixed. 
To describe the coefficient of its main term, put
\[
 H_N^{(w)}
 := \sum_{m=1}^{N}\frac1{m^w},
 \qquad
 H_0^{(w)}:=0,
\]
and, for a complex variable $\rho$, define
\begin{align}
 \zeta_2^{[2]}(s_1,\rho)
 &:= \sum_{n=1}^{\infty} \frac{|H_{n-1}^{(s_1)}|^2}{n^\rho}.
 \label{eq:intro-diagonal-coefficient}
\end{align}
The defining series converges precisely when $\Re\rho>1,~2\Re s_1+\Re\rho>3$.
The results of Matsumoto and Tsumura
\cite{MatsumotoTsumura}, together with the extension by
Ikeda, Matsuoka and Nagata
\cite[Theorem~1.2]{IkedaMatsuokaNagata}, give the following
formula throughout this range, provided that the path of
integration does not pass through a pole of $\zeta_2$:
\begin{align}
 \int_2^T|\zeta_2(s_1,\sigma_2+it)|^2dt
 &= \zeta_2^{[2]}(s_1,2\sigma_2)T+o(T).
 \label{eq:intro-MT-form}
\end{align}
Since the second argument on the right-hand side is $\rho=2\sigma_2$, the conditions $\Re\rho>1,~2\Re s_1+\Re\rho>3$ become
\begin{align}
 \sigma_2>\frac12,
 \qquad
 \Re s_1+\sigma_2>\frac32.
 \label{eq:intro-diagonal-region}
\end{align}
This led Matsumoto and Tsumura to suggest that the line $\Re s_1+\sigma_2=3/2$ should play a role analogous to that of the critical line for the Riemann zeta-function.

Ikeda, Matsuoka and Nagata \cite{IkedaMatsuokaNagata} subsequently investigated three different types of mean values of the double zeta-function and introduced corresponding approximation formulas. 
Among them, the case in which the second variable moves vertically is particularly relevant to the present paper.
Put $\sigma_1=\Re s_1$. 
Assuming that the path of integration does not meet the polar locus, their results on the boundary can be summarized as follows:
\begin{align}
 &\int_2^T |\zeta_2(s_1,\sigma_2+it)|^2dt\notag\\
 &\qquad=
 \begin{cases}
  |\zeta(s_1)|^2T\log T+O_{s_1}(T),
    &\sigma_1>1,~
    \sigma_2=\frac12,
   \\[8pt]
  \left(
   |\zeta(s_1)|^2+\dfrac1{|s_1-1|^2}
  \right)T\log T+O_{s_1}(T),
    &\sigma_1=1,~ s_1\neq1,~\sigma_2=\frac12,
   \\[12pt]
  \dfrac1{|s_1-1|^2}T\log T+O_{s_1}(T),
    &\sigma_1+\sigma_2=\frac32,~\sigma_2>\frac12,
   \\[12pt]
  \dfrac13T(\log T)^3
  +O\left(T(\log T)^2\right),
   &
   s_1=1,~ \sigma_2=\frac12.
 \end{cases}
 \label{eq:intro-IMN-boundary}
\end{align}

Mean-square formulas in wider parts of the critical region were obtained by Kiuchi and Minamide
\cite{KiuchiMinamide}. 
Related mean-value problems for the double zeta-function were also studied by Ikeda, Kiuchi and Matsuoka \cite{IkedaKiuchiMatsuoka}. 
The error estimates in the formulas of Kiuchi and Minamide were later improved by Banerjee, Minamide
and Tanigawa \cite{BanerjeeMinamideTanigawa}, using refined estimates for the remainder in an analytic decomposition of the double zeta-function.

In the present paper, we refine formulas \eqref{eq:intro-IMN-boundary} by determining the terms of order $T$.
For the case $\sigma_1=1,~ s_1\neq1,~\sigma_2=1/2$, this includes an oscillatory term of order $T$, which is contained in the $O(T)$-term of the preceding result. 
We also investigate the arithmetic structure of the coefficients and their relation to multiple zeta values and harmonic regularizations.

Our first observation is that the coefficient $\zeta_2^{[2]}(s_1,2\sigma_2)$ itself is a linear combination of Euler--Zagier multiple zeta-functions.
More generally, we use the increasing-index convention
\[
 \zeta_r(s_1,\ldots,s_r)
 := \sum_{1\leq n_1<\cdots<n_r} \frac{1}{n_1^{s_1}\cdots n_r^{s_r}},
\]
initially in its region of absolute convergence, and denote its meromorphic continuation by the same symbol.

\begin{theorem}\label{thm:intro-diagonal-MZV}
Let $s_1,\rho\in\C$, and put $\sigma_1=\Re s_1$. 
Suppose that $\Re\rho>1,~2\sigma_1+\Re\rho>3$.
Then the defining series for $\zeta_2^{[2]}(s_1,\rho)$ converges absolutely, and
\begin{align}
\zeta_2^{[2]}(s_1,\rho) = \zeta_3(s_1,\overline{s_1},\rho)
+\zeta_3(\overline{s_1},s_1,\rho)
+\zeta_2(2\sigma_1,\rho).
\label{eq:intro-diagonal-MZV}
\end{align}
If $\rho=2\sigma_2$, where $\sigma_2\in\mathbb R$, the above conditions become $\sigma_2>1/2$ and
$\sigma_1+\sigma_2>3/2$.
Assume in addition that the path $(s_1,\sigma_2+it)$, $t\geq2$, does not pass through a pole
of $\zeta_2$.
Then the mean-square formula \eqref{eq:intro-MT-form} gives
\begin{align*}
\int_2^T \left| \zeta_2(s_1,\sigma_2+it) \right|^2dt
=\Bigl(
\zeta_3(s_1,\overline{s_1},2\sigma_2) + \zeta_3(\overline{s_1},s_1,2\sigma_2) + \zeta_2(2\sigma_1,2\sigma_2)
\Bigr)T + o(T).
\end{align*}
In particular, if $k\geq1$ and $\ell\geq2$ are integers, then
\begin{align*}
\zeta_2^{[2]}(k,\ell) = 2\zeta_3(k,k,\ell)+\zeta_2(2k,\ell),
\end{align*}
and hence
\begin{align*}
\int_2^T \left| \zeta_2\left(k,\frac{\ell}{2}+it\right) \right|^2 dt
= \left( 2\zeta_3(k,k,\ell)+\zeta_2(2k,\ell) \right)T + o_{k,\ell}(T).
\end{align*}
Thus, at these integral and half-integral points, the coefficient of $T$ is a linear combination of convergent multiple zeta values.
\end{theorem}

We investigate the cases in which one of the inequalities in \eqref{eq:intro-diagonal-region} becomes an equality, together with their common intersection $(\Re s_1,\sigma_2)=\left(1,1/2\right)$.
For a fixed $a\in\mathbb C$ with $\Re a>0$, define
\begin{align*}
 C_{+}(a)
 &:= |\zeta(1+a)|^2(2\gamma-1) +\sum_{n=1}^{\infty} \frac{  |H_{n-1}^{(1+a)}|^2-|\zeta(1+a)|^2 }{n},\\
 C_{-}(a)
 &:= \frac{2\gamma-1}{|a|^2} +\sum_{n=1}^{\infty}
 \frac{  n^{-2\Re a}|H_{n-1}^{(1-a)}|^2-|a|^{-2} }{n}.
\end{align*}
Both series are absolutely convergent (see Remarks \ref{rem:plus-renormalized-diagonal} and \ref{Remark4.7}).
For fixed $b\in\mathbb R\setminus\{0\}$,  we set
\begin{align*}
 &\mathcal D_b
 = (2\gamma-1)\left(|\zeta(1+ib)|^2+b^{-2}\right)
 +\sum_{n=1}^{\infty} \frac{  |H_{n-1}^{(1+ib)}|^2  -  \left|\zeta(1+ib)+\frac{i}{b}n^{-ib}\right|^2 }{n},\\
 &\mathcal E_b(X)
 := -2\operatorname{Re}\left\{  \frac{i \zeta(1+ib)^2}{b(1+ib)}e^{ibX} \right\}.
\end{align*}
The series occurring in the definition of $\mathcal D_b$ converges absolutely, as will be shown in the proof of Theorem~\ref{thm:intro-boundary-main}.
We also put
\begin{align*}
 &C_0
 := -2+4\gamma-3\gamma^2+\frac43\gamma^3 +2(1-\gamma)\gamma_1+\gamma_2 -\frac43\zeta(3),\\
 &P_0(X)
 := \frac13X^3+(2\gamma-1)X^2 +(3\gamma^2-4\gamma+2-2\gamma_1)X+C_0,
\end{align*}
where the Stieltjes constants are defined by
\begin{align}
   \zeta(1+w)
 = \frac1w+ \sum_{j=0}^{\infty} \frac{(-1)^j\gamma_j}{j!}w^j.
 \label{eq:Stieltjes-convention-prelim}
\end{align}

For $a\in\mathbb C$ with $\Re a>0$, put $T_a:=2+|\Im a|$.

\begin{theorem}\label{thm:intro-boundary-main}
Let $a\in\mathbb C$ be fixed with $\Re a>0$, and let
$b\in\mathbb R\setminus\{0\}$ be fixed. 
Then, as $T\to\infty$, the following formulas hold:
\begin{align*}
 &\int_{T_a}^{T} \left|  \zeta_2\left(1+a,\frac12+it\right) \right|^2dt
 = |\zeta(1+a)|^2T\log\frac{T}{2\pi} +C_{+}(a)T+o_a(T),\\
 &\int_2^T \left| \zeta_2\left(1+ib,\frac12+it\right) \right|^2dt
 = \left(|\zeta(1+ib)|^2+b^{-2}\right)T\log\frac{T}{2\pi} +\mathcal D_bT +T\mathcal E_b\left(\log\frac{T}{2\pi}\right) +o_b(T),\\
 &\int_2^T \left| \zeta_2\left(1,\frac12+it\right) \right|^2dt
 = T P_0\left(\log\frac{T}{2\pi}\right) +o(T), \\
 &\int_{T_a}^{T} \left| \zeta_2\left(1-a,\frac12+a+it\right) \right|^2dt
 = \frac{T}{|a|^2}\log\frac{T}{2\pi} +C_{-}(a)T+o_a(T).  
\end{align*}
\end{theorem}

\begin{remark}
Each of the definitions of $C_{+}(a)$, $C_{-}(a)$, and $\mathcal D_b$ contains an absolutely convergent infinite series. 
These series can be recovered from the constant terms in the Laurent expansions of the corresponding diagonal Dirichlet series 
$\zeta_2^{[2]}(s_1,\rho)=\zeta_3(s_1,\overline{s_1},\rho)+\zeta_3(\overline{s_1},s_1,\rho)+\zeta_2(2\Re s_1,\rho)$ 
at their boundary points. 
More precisely, the relevant point is $\rho=1$ for $C_{+}(a)$ and $\mathcal D_b$, whereas for $C_{-}(a)$ it is $\rho=1+2\Re a$, or equivalently $\rho=1$ after considering $\zeta_2^{[2]}(1-a,\rho+2\Re a)$.
They may also be viewed as suitably normalized harmonic finite parts of the corresponding divergent diagonal sums. 
The associated Laurent constants may therefore be regarded as Euler-type constants attached to these diagonal Dirichlet series. 
See Remarks~\ref{Remark4.2}, \ref{Remark4.3}, \ref{Remark4.5}, and \ref{Remark4.7} for further details.
\end{remark}

At positive integral points and at the corner, the lower-order coefficients in Theorem~\ref{thm:intro-boundary-main} can also be interpreted arithmetically. 
Let $\zeta^{*}(k_1,\ldots,k_r)$ denote the constant term of the harmonic regularization of the multiple zeta value with index $(k_1,\ldots,k_r)$. 

For integral parameters, the finite parts of the divergent diagonal sums are precisely harmonic regularized multiple zeta values.

\begin{theorem}\label{thm:intro-regularized-main}
For every integer $k\geq2$, one has
\begin{align*}
 C_{+}(k-1)
 &= \zeta(k)^2(2\gamma-1) + 2\zeta^{*}(k,k,1) + \zeta^{*}(2k,1). 
\end{align*}
Moreover,  using the relation $2\zeta^{*}(1,1,1)+\zeta^{*}(2,1)=-4\zeta(3)/3$, we have
\begin{align*}
 C_0
 &= -2+4\gamma-3\gamma^2+\frac43\gamma^3 +2(1-\gamma)\gamma_1+\gamma_2 +2\zeta^{*}(1,1,1)+\zeta^{*}(2,1).
\end{align*}
\end{theorem}

It is worth noting the change of sign at the boundary. 
For integers $k\geq2$ and $\ell\geq2$, the convergent multiple zeta values $\zeta_3(k,k,\ell)$ and $\zeta_2(2k,\ell)$ are positive real numbers.
In contrast, their harmonic regularized boundary counterparts $\zeta^{*}(k,k,1)$ and $\zeta^{*}(2k,1)$ are negative real numbers.
Thus harmonic regularization does not, in general, preserve positivity.

\begin{remark}\label{rem:integral-half-integral-mean-square}
Combining \eqref{eq:intro-MT-form} with Theorem~\ref{thm:intro-diagonal-MZV}, we find that, for integers $k\geq1$ and $\ell\geq2$,
\begin{align}
 \int_2^T \left| \zeta_2\left(k,\frac{\ell}{2}+it\right) \right|^2dt
 &= \left\{ 2\zeta_3(k,k,\ell)+\zeta_2(2k,\ell) \right\}T
 +o_{k,\ell}(T).
 \label{eq:intro-integral-half-integral-mean-square}
\end{align}
Since $\ell\geq2$, all the multiple zeta values on the right-hand side are convergent.

For $k\geq2$, the point $(s_1,\sigma_2)=(k,1/2)$ formally corresponds to setting $\ell=1$ in \eqref{eq:intro-integral-half-integral-mean-square}, but the resulting multiple zeta series $2\zeta_3(k,k,1)+\zeta_2(2k,1)$ are divergent.
Theorems~\ref{thm:intro-boundary-main} and \ref{thm:intro-regularized-main} express the lower-order coefficient in terms of their harmonic regularizations and give
\begin{align*}
 &\int_2^T \left| \zeta_2\left(k,\frac12+it\right) \right|^2dt\\
 &\qquad= \zeta(k)^2T\log\frac{T}{2\pi}
 +\left\{ \zeta(k)^2(2\gamma-1) +2\zeta^{*}(k,k,1) +\zeta^{*}(2k,1) \right\}T
 +o_k(T).
\end{align*}
Here $\zeta(k)^2(2\gamma-1)$ is obtained by multiplying the lower-order coefficient in the classical mean-square formula for $\zeta(1/2+it)$ by $\zeta(k)^2$, whereas $2\zeta^{*}(k,k,1)+\zeta^{*}(2k,1)$ is the harmonic finite part of the divergent diagonal multiple zeta sums formally represented by $2\zeta_3(k,k,1)+\zeta_2(2k,1)$.

When $k=1$, the two boundary lines meet at the corner $(s_1,\sigma_2)=\left(1,1/2\right)$.
In this case, the mean-square formula takes the form
\begin{align*}
 \int_2^T \left| \zeta_2\left(1,\frac12+it\right) \right|^2dt
 &= T\left\{ \frac13X^3+(2\gamma-1)X^2 +(3\gamma^2-4\gamma+2-2\gamma_1)X+C_0 \right\}
 +o(T),
\end{align*}
where $X=\log\frac{T}{2\pi}$.
The formal specialization $k=\ell=1$ of the diagonal multiple-zeta decomposition gives  $2\zeta_3(1,1,1)+\zeta_2(2,1)$.
Its harmonic finite part is $2\zeta^{*}(1,1,1)+\zeta^{*}(2,1) = -4\zeta(3)/3$, and Theorem~\ref{thm:intro-regularized-main} shows that it occurs in the constant term
\begin{align*}
 C_0
 = -2+4\gamma-3\gamma^2+\frac43\gamma^3+2(1-\gamma)\gamma_1+\gamma_2 +2\zeta^{*}(1,1,1)+\zeta^{*}(2,1).
\end{align*}
\end{remark}

\begin{remark}\label{rem:C-zero-multiple-Stieltjes}
Matsumoto, Onozuka and Wakabayashi \cite{MatsumotoOnozukaWakabayashi} defined multiple Stieltjes constants through the Laurent expansion of the Euler--Zagier multiple zeta-function at $(1,\ldots,1)$. 
In their normalization, setting $s_1=s_2=1$ and $s_3=1+u$ gives
\begin{align}
\zeta_3(1,1,1+u)
&= \frac1{u^3} +\frac{\gamma_{(0,0,0)}}{u^2} +\frac{\gamma_{(0,0,1)}}u +\gamma_{(0,0,2)}+O(u).
\label{eq:triple-Stieltjes-expansion}
\end{align}
Their double Euler constant formula also gives
\begin{align}
\zeta_2(2,1+u)
&=\frac{\zeta(2)}u+\gamma\zeta(2)-2\zeta(3)+O(u).
\label{eq:double-Euler-expansion-at-two}
\end{align}

On the other hand, by partial summation,
\[
 \zeta_2^{[2]}(1,1+u)
 = u\int_1^\infty \left(\sum_{n\leq x}\frac{H_{n-1}^2}{n}\right) x^{-1-u}dx
 \qquad(\Re u>0),
\]
and inserting \eqref{eq:corner-H-square-asymptotic} below, together with $u\int_1^\infty(\log x)^kx^{-1-u}dx=k!\,u^{-k}$ $(\Re u>0)$, gives
\[
 \zeta_2^{[2]}(1,1+u)
 = \frac{2}{u^3}+\frac{2\gamma}{u^2}+\frac{\gamma^2}{u}
 +\frac{\gamma^3}{3}-\frac43\zeta(3)+o(1)
 \qquad(u\to0+).
\]
By Theorem~\ref{thm:intro-diagonal-MZV} and analytic continuation,
$\zeta_2^{[2]}(1,\rho)=2\zeta_3(1,1,\rho)+\zeta_2(2,\rho)$,
so $\zeta_2^{[2]}(1,\rho)$ is meromorphic at $\rho=1$ with a pole of order at most three.
Hence the asymptotic above determines its Laurent coefficients there.
Comparing the constant terms with \eqref{eq:triple-Stieltjes-expansion} and
\eqref{eq:double-Euler-expansion-at-two}, we obtain
\begin{align*}
2\gamma_{(0,0,2)} + \gamma\zeta(2) - 2\zeta(3) &= \frac{\gamma^3}{3}-\frac43\zeta(3),
\end{align*}
or equivalently,
\begin{align*}
\gamma_{(0,0,2)} &= \frac{\gamma^3}{6}-\frac{\gamma\zeta(2)}2+\frac{\zeta(3)}3.
\end{align*}
Consequently, the constant $C_0$ may also be written as
\begin{align*}
C_0 =-2+4\gamma-3\gamma^2+2\gamma^3+2(1-\gamma)\gamma_1+\gamma_2-2\gamma\zeta(2)-4\gamma_{(0,0,2)}.
\end{align*}
Thus the constant term in the mean-square formula at the corner can also be expressed in terms of the multiple Stieltjes constants introduced in the author's previous joint work \cite{MatsumotoOnozukaWakabayashi}.
\end{remark}

The paper is organized as follows. 
In Section~\ref{sec:analytic-preliminaries}, we derive the required Euler--Maclaurin decompositions and collect mean-value formulas for the Riemann zeta-function and absolutely convergent Dirichlet series. 
In Section~\ref{sec:diagonal-coefficient}, we prove the multiple-zeta decomposition of $\zeta_2^{[2]}(s_1,2\sigma_2)$ stated in Theorem~\ref{thm:intro-diagonal-MZV}. 
In Section~\ref{sec:boundary-mean-square}, we establish the mean-square formulas for the three boundary cases stated in Theorem~\ref{thm:intro-boundary-main}, treating the middle case separately according as $b\neq0$ or $b=0$. 
Finally, in Section~\ref{sec:regularized-MZV}, we interpret the relevant lower-order coefficients at positive integral points of the first case and at the corner in terms of harmonic regularized multiple zeta values.

\section{Analytic preliminaries}\label{sec:analytic-preliminaries}

In this section, we collect the analytic tools used in the proofs of our mean-square formulas. We first recall some consequences of the Euler--Maclaurin formula and then give the mean-value formulas for the Riemann zeta-function and absolutely convergent Dirichlet series that will be needed later.

It is known that the Euler--Zagier multiple zeta-function admits a meromorphic continuation to the whole complex space. 
The meromorphic continuation was established by Zhao \cite{Zhao} and, independently, by Akiyama, Egami and Tanigawa \cite{AkiyamaEgamiTanigawa}. 
In particular, Akiyama, Egami and Tanigawa obtained the continuation by applying the Euler--Maclaurin summation formula recursively.

The decompositions used below should not be regarded as new meromorphic continuation formulas. They are convenient specializations of the Euler--Maclaurin method to the three cases considered in this paper. 
Their advantage for our purpose is that they separate the relevant Riemann zeta-functions from Dirichlet series which converge absolutely on the critical line.

Throughout the paper, we write $s_t:=1/2+it$ and $X:=\log(T/(2\pi))$.

\subsection{Euler--Maclaurin decompositions}

In this subsection, we record four Euler--Maclaurin decompositions.  
These formulas isolate the relevant Riemann zeta-functions from Dirichlet series which converge absolutely on the critical line.

For $w\in\mathbb{C}$, initially in the region of absolute convergence,
\begin{align}
 \zeta_2(w,s)
 &= \sum_{n=1}^{\infty} \frac{H_{n-1}^{(w)}}{n^s}. 
 \label{eq:double-zeta-harmonic}
\end{align}

\begin{lemma}\label{lem:Euler-Maclaurin-harmonic-expansions}
Let $a\in\mathbb C$ be fixed with $\Re a>0$, and let $b\in\mathbb R\setminus\{0\}$ be fixed. 
We have
\begin{align}
&H_{n-1}^{(1+a)}
= \zeta(1+a)-\frac{n^{-a}}a+r_a^{+}(n),
\label{eq:H-plus-AET}\\
&H_{n-1}^{(1+ib)}
= \zeta(1+ib)-\frac{n^{-ib}}{ib}+r_b(n),
\label{eq:H-imaginary-AET}\\
&H_{n-1}
= \log n+\gamma+r_0(n),
\label{eq:H-zero-AET}\\
&n^{-a}H_{n-1}^{(1-a)}
= \frac1a+\zeta(1-a)n^{-a}+r_a^{-}(n),
\label{eq:H-minus-AET}
\end{align}
where
\begin{align}
&r_a^{+}(n)
= -\frac12n^{-1-a}+O_a\left(n^{-2-\Re a}\right),
\label{eq:r-plus-AET}\\
&r_b(n)
= -\frac12n^{-1-ib}+O_b(n^{-2}),
\label{eq:r-imaginary-AET}\\
&r_0(n)
= -\frac1{2n}+O(n^{-2}),
\label{eq:r-zero-AET}\\
&r_a^{-}(n)
= -\frac1{2n}+O_a(n^{-2}).
\label{eq:r-minus-AET}
\end{align}
\end{lemma}

\begin{proof}
The Euler--Maclaurin formula, in the form used in \cite{AkiyamaEgamiTanigawa}, gives, for every fixed $w\in\mathbb C\setminus\{1\}$,
\begin{align*}
H_{n-1}^{(w)}
&=\zeta(w)+\frac{n^{1-w}}{1-w}-\frac12n^{-w}+O_w\left(n^{-\Re w-1}\right).
\end{align*}
Applying this formula with $w=1+a$, $w=1+ib$, and $w=1-a$, and multiplying the last result by $n^{-a}$, gives \eqref{eq:H-plus-AET}, \eqref{eq:H-imaginary-AET}, and \eqref{eq:H-minus-AET}, together with their stated remainder estimates. 
Finally, the classical expansion $H_{n-1}=\log n+\gamma-(2n)^{-1}+O(n^{-2})$ gives \eqref{eq:H-zero-AET}.
\end{proof}

Define
\begin{align*}
 &B_a^{+}(s)
 := \sum_{n=1}^{\infty}\frac{r_a^{+}(n)}{n^s}, 
 &B_b(s)
 := \sum_{n=1}^{\infty}\frac{r_b(n)}{n^s}, \\
 &B_0(s)
 := \sum_{n=1}^{\infty}\frac{r_0(n)}{n^s}, 
 &B_a^{-}(s)
 := \sum_{n=1}^{\infty}\frac{r_a^{-}(n)}{n^s}.
\end{align*}
It follows from the estimates above that $B_a^{+}(s)$ converges absolutely for $\Re s>-\Re a$ while $B_b(s)$, $B_0(s)$, and $B_a^{-}(s)$ converge absolutely for $\Re s>0$.

\begin{proposition}\label{prop:boundary-decompositions}
Let $a\in\mathbb C$ and $b\in\mathbb R$ be fixed, with $\Re a>0$ and $b\neq0$. 
Away from their polar loci, \eqref{eq:decomposition-plus} below holds for $\Re s>-\Re a$, while \eqref{eq:decomposition-imaginary}--\eqref{eq:decomposition-minus} hold for $\Re s>0$:
\begin{align}
 &\zeta_2(1+a,s)
 = \zeta(1+a)\zeta(s) -\frac1a\zeta(s+a) +B_a^{+}(s), 
 \label{eq:decomposition-plus}\\
 &\zeta_2(1+ib,s)
 = \zeta(1+ib)\zeta(s) -\frac1{ib}\zeta(s+ib) +B_b(s), 
 \label{eq:decomposition-imaginary}\\
 &\zeta_2(1,s)
 = -\zeta'(s)+\gamma\zeta(s)+B_0(s), 
 \label{eq:decomposition-zero}\\
 &\zeta_2(1-a,s+a)
 = \frac1a\zeta(s) +\zeta(1-a)\zeta(s+a) +B_a^{-}(s). 
 \label{eq:decomposition-minus}
\end{align}
\end{proposition}

\begin{proof}
We first assume that $\Re s>1$. 
Substituting \eqref{eq:H-plus-AET} into \eqref{eq:double-zeta-harmonic}, we obtain
\begin{align*}
\zeta_2(1+a,s)
&=\sum_{n=1}^{\infty} \frac{\zeta(1+a)-a^{-1}n^{-a}+r_a^{+}(n)}{n^s}\\
&=\zeta(1+a)\zeta(s)-\frac1a\zeta(s+a)+B_a^{+}(s).
\end{align*}
This proves \eqref{eq:decomposition-plus} in the region $\Re s>1$.

The other three identities follow in the same way. 
Indeed, using \eqref{eq:H-imaginary-AET}, \eqref{eq:H-zero-AET}, and \eqref{eq:H-minus-AET}, respectively, \eqref{eq:decomposition-imaginary}, \eqref{eq:decomposition-zero}, and \eqref{eq:decomposition-minus} also hold for $\Re s>1$.

By the results of Zhao \cite{Zhao} and Akiyama, Egami and Tanigawa \cite{AkiyamaEgamiTanigawa}, the functions on the left-hand sides have meromorphic continuations as functions of $s$. 
On the right-hand sides, the Riemann zeta-functions are meromorphic, while the remainder series are holomorphic in the respective half-planes stated above. 
Since the two sides agree in the nonempty open half-plane $\Re s>1$, the identity theorem for meromorphic functions extends \eqref{eq:decomposition-plus} to $\Re s>-\Re a$, and extends \eqref{eq:decomposition-imaginary}, \eqref{eq:decomposition-zero}, and \eqref{eq:decomposition-minus} to $\Re s>0$.
The asserted identities follow at every point in these domains at which both sides are regular.
\end{proof}

\subsection{Mean values of the Riemann zeta-function}
\label{subsec:mean-values-riemann-zeta}

We collect here the mean-value formulas for the Riemann zeta-function.
The classical mean-square formula \cite[Theorem~7.4 and the remark following its proof]{Titchmarsh} is
\begin{equation}
\label{eq:classical-zeta-mean-square}
 \int_2^T \left|\zeta\left(\frac12+it\right)\right|^2dt
 = T\log\frac{T}{2\pi} +(2\gamma-1)T +O\left(T^{1/2}\log^2T\right).
\end{equation}
For our purposes, we shall also need a version with fixed complex shifts. 
Ingham \cite[Theorem~A]{Ingham} considers the integral
$$
 \int_{\kappa}^{T} \zeta(\alpha+it)\zeta(\beta-it)\,dt.
$$
To express Ingham's formula in the notation used here, we put $\alpha=1/2+u,~\beta=1/2+\overline v$.
Then $\alpha+\beta=1+u+\overline v$.
The factor $1/2$ in the following lemma therefore comes only from our choice to measure the shifts from the critical line.

\begin{lemma}
\label{lem:fixed-shifted-second-moment}
Let $u,v\in\mathbb C$ be fixed, with $\Re u\geq0,~\Re v\geq0$, and put $c:=u+\overline v$.
Suppose that $c\neq0$, and choose $T_0\geq2$ so that the paths of integration contain no pole. 
Then
\begin{align*}
\int_{T_0}^{T} \zeta\left(\frac12+u+it\right)\overline{\zeta\left(\frac12+v+it\right)} dt
= T\zeta(1+c) +\zeta(1-c)
\int_{T_0}^{T} \left(\frac{t}{2\pi}\right)^{-c} dt
+o_{u,v}(T).
\end{align*}
\end{lemma}

\begin{proof}
Since $\overline{ \zeta\left(1/2+v+it\right) } = \zeta\left(1/2+\overline v-it\right)$, putting $\alpha:=1/2+u,~\beta:=1/2+\overline v$, Ingham's Theorem~A gives
\begin{align*}
 &\int_{\kappa}^{T} \zeta(\alpha+it)\zeta(\beta-it)\,dt =
 2\pi F\left(\frac{T}{2\pi},1+c\right) +o_{u,v}(T),
\end{align*}
where
\[
 F(Y,1+c)
 = \int_1^Y \left\{ \zeta(1+c)+\tau^{-c}\zeta(1-c) \right\} d\tau.
\]
Consequently, after making the change of variables $t=2\pi\tau$, we obtain
\begin{align*}
 2\pi F\left(\frac{T}{2\pi},1+c\right)
 = (T-2\pi)\zeta(1+c)+ \zeta(1-c)
 \int_{2\pi}^{T} \left(\frac{t}{2\pi}\right)^{-c} dt.
\end{align*}
Since $\kappa$, $T_0$, and $2\pi$ are fixed, and the corresponding finite paths contain no pole, changing the lower endpoints on both sides to $T_0$ contributes only $O_{u,v}(1)$.
Moreover, $2\pi\zeta(1+c)=O_{u,v}(1)$.
These $O_{u,v}(1)$-terms are absorbed into $o_{u,v}(T)$, and the asserted formula follows.
\end{proof}

When the combined shift has positive real part, the second term in Lemma \ref{lem:fixed-shifted-second-moment} is of lower order.

\begin{corollary}
\label{cor:fixed-shifted-second-moment-positive}
Under the assumptions of Lemma~\ref{lem:fixed-shifted-second-moment}, suppose in addition
that $\Re(u+\overline v)>0$.
Then
\begin{align*}
 &\int_{T_0}^{T} \zeta\left(\frac12+u+it\right) \overline{ \zeta\left(\frac12+v+it\right) }\,dt
 = T\zeta(1+u+\overline v)+o_{u,v}(T).
\end{align*}
\end{corollary}

\begin{proof}
Put $c=u+\overline v$. 
If $c\neq1$, then
\[
 \int_{T_0}^{T} \left(\frac{t}{2\pi}\right)^{-c} dt
 = (2\pi)^c \frac{T^{1-c}-T_0^{1-c}}{1-c},
\]
whereas for $c=1$ this integral is $2\pi\log(T/T_0)$.
Thus, in either case, this integral is $o_c(T)$ because $\Re c>0$. 
\end{proof}

We record the following consequence for use in the middle boundary case.

\begin{corollary}
\label{cor:pure-imaginary-shift}
Let $b\in\mathbb R\setminus\{0\}$ be fixed.
Then
\begin{align}
 &\int_2^T \zeta\left(\frac12+it\right)
 \overline{ \zeta\left(\frac12+i(t+b)\right) }\,dt
 = T\zeta(1-ib) +\frac{\zeta(1+ib)}{1+ib} T\left(\frac{T}{2\pi}\right)^{ib}
 +o_b(T). 
 \label{eq:pure-imaginary-shift}
\end{align}
Moreover,
\begin{align}
 \int_2^T \left| \zeta\left(\frac12+i(t+b)\right) \right|^2dt
 &= T\log\frac{T}{2\pi} +(2\gamma-1)T+o_b(T).
 \label{eq:pure-imaginary-shift-square}
\end{align}
\end{corollary}

\begin{proof}
In Lemma~\ref{lem:fixed-shifted-second-moment}, take $u=0,~ v=ib,~ c=-ib$.
Since
\begin{align*}
 \int_2^T \left(\frac{t}{2\pi}\right)^{ib} dt
 &= \frac{T}{1+ib} \left(\frac{T}{2\pi}\right)^{ib}+O_b(1),
\end{align*}
formula \eqref{eq:pure-imaginary-shift} follows from Lemma \ref{lem:fixed-shifted-second-moment}.  The second assertion follows from \eqref{eq:classical-zeta-mean-square} after the change of variable $u=t+b$.
\end{proof}

\begin{remark}
\label{rem:fixed-shift-secondary-term}
When $\Re(u+\overline v)>0$, the secondary term in Ingham's formula is absorbed into $o(T)$. 
For the purely imaginary shift in Corollary~\ref{cor:pure-imaginary-shift}, however, it is of order $T$ and produces the oscillatory factor $(T/(2\pi))^{ib}$.
\end{remark}

To evaluate the mean square involving $\zeta'(s)$ at the corner, we shall differentiate a shifted second-moment formula with respect to the shifts.  
For this purpose, we require a version that is uniform near the origin.
For complex variables $\alpha,\beta$, put
\[
 \mathcal M(\alpha,\beta;T)
 := \int_2^T \zeta\left(\frac12+\alpha+it\right) \zeta\left(\frac12+\beta-it\right)\,dt.
\]

\begin{lemma}
\label{lem:uniform-shifted-second-moment}
Let $\eta>0$ be fixed. 
Uniformly for $|\alpha|\leq\eta/\log T$ and $|\beta|\leq\eta/\log T$, we have
\begin{equation}
\label{eq:uniform-shifted-second-moment}
 \mathcal M(\alpha,\beta;T)
 = TQ(w;X) +O_\eta\left(T^{1/2}\log^2T\right),
\end{equation}
where $w=\alpha+\beta$ and
\begin{equation}
\label{eq:definition-Q}
 Q(w;X)
 := \zeta(1+w) + \frac{e^{-wX}}{1-w}\zeta(1-w).
\end{equation}
At $w=0$, the right-hand side of \eqref{eq:definition-Q} is interpreted by continuity.
\end{lemma}

\begin{proof}
Bettin's uniform shifted second-moment formula \cite[Theorem~1]{Bettin}, applied with his parameters, which we denote $a=\alpha$ and $b=-\beta$, gives
\begin{align}
 \mathcal M(\alpha,\beta;T)
 &= \int_2^T
 \Biggl\{ \zeta(1+w) + \zeta(1-w) \chi\left(\frac12+\alpha+it\right) \chi\left(\frac12+\beta-it\right)
 \Biggr\} dt \notag\\
 &\quad
 +O_\eta\left(T^{1/2}\log^2T\right)
 \label{eq:Bettin-shifted-form}
\end{align}
where $\chi(s) := \pi^{s-1/2} \Gamma\left(\frac{1-s}{2}\right)/\Gamma\left(\frac{s}{2}\right)$, so that $\zeta(s)=\chi(s)\zeta(1-s)$.
Changing the lower endpoint in Bettin's formula from $0$ to $2$ contributes $O_\eta(1)$, which is absorbed into the error term above.

We next remove the two $\chi$-factors. 
Put $z:=1/2+\alpha+it$.
Since $1/2+\beta-it=1-z+w$ and $\chi(z)\chi(1-z)=1$,
\begin{align}
 \chi\left(\frac12+\alpha+it\right) \chi\left(\frac12+\beta-it\right)
 = \frac{\chi(1-z+w)}{\chi(1-z)}
 =
 \exp\left\{ w\int_0^1\frac{\chi'}{\chi}(1-z+\theta w)\,d\theta \right\}.
 \label{eq:chi-ratio}
\end{align}
Here $1-z+\theta w$ has bounded real part $1/2+O(1/\log T)$ and imaginary part $-t+O(1/\log T)$, while logarithmic differentiation of $\chi(s)$ gives
\[
 \frac{\chi'}{\chi}(s)
 = \log\pi -\frac12\psi\left(\frac{1-s}2\right) -\frac12\psi\left(\frac s2\right),
 \qquad\psi:=\frac{\Gamma'}{\Gamma}.
\]
Both arguments have modulus $\asymp t$ with opposite imaginary parts, so $\psi(\xi)=\log\xi+O\left(|\xi|^{-1}\right)$ applies to each and the two terms $\pm i\pi/2$ cancel, leaving
\[
 \frac{\chi'}{\chi}(1-z+\theta w)
 = -\log\frac t{2\pi}+O_\eta\left(\frac1t\right)
 \qquad(2\leq t\leq T,\ 0\leq\theta\leq1),
\]
uniformly.  
Inserting this into \eqref{eq:chi-ratio} yields
\begin{align}
 \chi\left(\frac12+\alpha+it\right) \chi\left(\frac12+\beta-it\right)
 = \left(\frac t{2\pi}\right)^{-w} \left\{1+O_\eta\left(\frac{|w|}t\right)\right\}.
 \label{eq:chi-product-approximation}
\end{align}

Since $|w\zeta(1-w)|\ll_\eta1$ and $|(t/(2\pi))^{-w}|\ll_\eta1$, the contribution of the error term in \eqref{eq:chi-product-approximation} to \eqref{eq:Bettin-shifted-form} is $ \ll_\eta \int_2^T dt/t \ll_\eta\log T$.
This is absorbed into $O_\eta(T^{1/2}\log^2T)$. 
We therefore obtain
\begin{align*}
 \mathcal M(\alpha,\beta;T)
 &= \int_2^T \left\{ \zeta(1+w) + \zeta(1-w) \left(\frac{t}{2\pi}\right)^{-w} \right\} dt
 +O_\eta\left(T^{1/2}\log^2T\right).
\end{align*}

For $w\neq0$, direct integration gives
\begin{align*}
 &\int_2^T \left\{ \zeta(1+w) + \zeta(1-w) \left(\frac{t}{2\pi}\right)^{-w} \right\} dt
 = T\left\{ \zeta(1+w) + \frac{e^{-wX}}{1-w}\zeta(1-w) \right\} +R(w),
\end{align*}
where
\[
 R(w) := -2\zeta(1+w) - \frac{(2\pi)^w2^{1-w}}{1-w}\zeta(1-w).
\]
Although the two terms defining $R(w)$ have apparent singularities at $w=0$, their residues cancel. 
Thus $R(w)$ is holomorphic in a neighborhood of $w=0$ and $R(w)\ll_\eta1$ for $|w|\leq2\eta/\log T$.
Hence we obtain \eqref{eq:uniform-shifted-second-moment}.
The formula for $w=0$ follows by continuity, since the singularities of the two terms defining $Q(w;X)$ cancel.
\end{proof}

We shall need derivatives of\eqref{eq:uniform-shifted-second-moment}. 
Using the convention\eqref{eq:Stieltjes-convention-prelim} and expanding\eqref{eq:definition-Q} at $w=0$, we find
\begin{align}
 Q(w;X)
 &= X+2\gamma-1 + \left\{ -\frac12X^2+(1-\gamma)X+\gamma-1 \right\}w \notag\\
 &\quad+
 \left\{ \frac16X^3+\frac{\gamma-1}{2}X^2+(1-\gamma-\gamma_1)X-1+\gamma+\gamma_1+\gamma_2 \right\}w^2 \notag\\
 &\quad+O_\eta\bigl((1+X^4)|w|^3\bigr).
\label{eq:Q-Taylor-expansion}
\end{align}

\begin{corollary}
\label{cor:mean-square-A}
As $T\to\infty$,
\begin{equation}
\label{eq:mean-square-A}
 \int_2^T \left|-\zeta'\left(s_t\right)+\gamma\zeta\left(s_t\right)\right|^2 dt
 = TP_{-\zeta'+\gamma\zeta}\left(\log\frac{T}{2\pi}\right)+o(T),
\end{equation}
where
\begin{align}
 P_{-\zeta'+\gamma\zeta}(X)
 &= \frac13X^3 + (2\gamma-1)X^2 +
 \left( 3\gamma^2-4\gamma+2-2\gamma_1 \right)X \notag\\
 &\quad -2+4\gamma-3\gamma^2+2\gamma^3+2\gamma_1+2\gamma_2.
\label{eq:polynomial-PA}
\end{align}
\end{corollary}

\begin{proof}
Since the estimate in Lemma~\ref{lem:uniform-shifted-second-moment} is uniform in a neighborhood of $(\alpha,\beta)=(0,0)$, Cauchy's integral formula allows us to differentiate it with respect to both shifts. 
From
\[
 -\zeta'\left(s_t+\alpha\right)+\gamma\zeta\left(s_t+\alpha\right)
 = \left(-\frac{\partial}{\partial\alpha}+\gamma\right) \zeta\left(s_t+\alpha\right),
\]
we obtain
\begin{align*}
 &\int_2^T \left|-\zeta'\left(s_t\right)+\gamma\zeta\left(s_t\right)\right|^2 dt
 = \left.
 \left(-\frac{\partial}{\partial\alpha}+\gamma\right) \left(-\frac{\partial}{\partial\beta}+\gamma\right) \mathcal M(\alpha,\beta;T)
 \right|_{\alpha=\beta=0}.
\end{align*}
Because $Q(\alpha+\beta;X)$ depends on the shifts only through $w=\alpha+\beta$, the main term on the right-hand side is $T\left\{ Q''(0;X)-2\gamma Q'(0;X)+\gamma^2Q(0;X) \right\}$.

To justify the differentiation of the error term, define $E(\alpha,\beta;T):=\mathcal M(\alpha,\beta;T) -TQ(\alpha+\beta;X)$.
Both $\mathcal M(\alpha,\beta;T)$ and $Q(\alpha+\beta;X)$, with the removable singularity at $\alpha+\beta=0$ filled in, are holomorphic in $(\alpha,\beta)$ in a neighborhood of $(0,0)$.
Moreover, Lemma~\ref{lem:uniform-shifted-second-moment} gives $E(\alpha,\beta;T) \ll_\eta T^{1/2}\log^2T$ uniformly for $|\alpha|,|\beta| \leq \eta/\log T$.
Put $r:=\eta/(2\log T)$.
Applying Cauchy's integral formula on the circles $|\alpha|=r$ and $|\beta|=r$, we obtain
\begin{align*}
 &E(0,0;T) \ll_\eta T^{1/2}\log^2T,\\
 &\frac{\partial E}{\partial\alpha}(0,0;T),~\frac{\partial E}{\partial\beta}(0,0;T)
 \ll_\eta r^{-1}T^{1/2}\log^2T
 \ll_\eta T^{1/2}\log^3T,
\end{align*}
and
\begin{align*}
 \frac{\partial^2E}{\partial\alpha\,\partial\beta}(0,0;T)
 &= \frac{1}{(2\pi i)^2}
 \int_{|u|=r}\int_{|v|=r} \frac{E(u,v;T)}{u^2v^2}\,dv\,du\\
 &\ll_\eta r^{-2}T^{1/2}\log^2T
 \ll_\eta T^{1/2}\log^4T.
\end{align*}
Therefore,
\begin{align*}
 &\left.
 \left(-\frac{\partial}{\partial\alpha}+\gamma\right)
 \left(-\frac{\partial}{\partial\beta}+\gamma\right)
 E(\alpha,\beta;T) \right|_{\alpha=\beta=0}\ll_\eta T^{1/2}\log^4T=o(T).
\end{align*}
Finally, substituting the coefficients from \eqref{eq:Q-Taylor-expansion} gives \eqref{eq:polynomial-PA}.
\end{proof}

\subsection{Mixed mean values with an absolutely convergent Dirichlet series}
\label{subsec:mixed-mean-values}

We next prepare mean-value formulas involving a Dirichlet series which is absolutely convergent on the critical line. 
The main analytic tool is the Montgomery--Vaughan mean-value theorem for Dirichlet polynomials \cite{MontgomeryVaughan}. 
We shall also use the classical truncated formula for the Riemann zeta-function \cite[Theorem~4.11]{Titchmarsh}. 
Similar diagonal and off-diagonal decompositions are used in the study of the mean square of the Euler--Zagier double zeta-function; see, for example, Matsumoto and Tsumura \cite{MatsumotoTsumura}.

For the convenience of readers less familiar with mean-value theorems for Dirichlet polynomials, we first recall the precise form that will be used below. 
If
\[
 P_{\boldsymbol a}(t) := \sum_{n\leq N}a_n n^{-it},
 \qquad
 P_{\boldsymbol b}(t) := \sum_{n\leq N}b_n n^{-it},
\]
then the Montgomery--Vaughan theorem, in its bilinear form, gives
\begin{align}
 \int_U^{U+L} P_{\boldsymbol a}(t) \overline{P_{\boldsymbol b}(t)}\,dt
 &=
 L\sum_{n\leq N}a_n\overline{b_n}+
 O\left( \left(\sum_{n\leq N}n|a_n|^2\right)^{1/2} \left(\sum_{n\leq N}n|b_n|^2\right)^{1/2} \right).
 \label{eq:MV-bilinear}
\end{align}
The first term on the right-hand side is the contribution from equal indices. 
Thus \eqref{eq:MV-bilinear} estimates all the off-diagonal terms collectively.

For a sequence $(d_n)_{n\geq1}$ of complex numbers satisfying $d_n\ll n^{-1}$, let $D(s):=\sum_{n=1}^{\infty} d_n n^{-s}$.
Then,
\begin{equation}
 \sup_{t\in\mathbb R} \left|D\left(s_t\right)\right|
 \leq \sum_{n=1}^{\infty}\frac{|d_n|}{\sqrt n}
 < \infty.
 \label{eq:D-uniform-bound}
\end{equation}

\begin{lemma}
\label{lem:mixed-Dirichlet-prelim}
Let $u\in\mathbb C$ be fixed with $\Re u\geq0$, and suppose that the path $1/2+u+it$ ($t\geq T_0$) does not pass through the pole $s=1$. 
Then
\begin{align}
 \int_{T_0}^{T} \zeta(s_t+u)\overline{D(s_t)}\,dt
 &= T\sum_{n=1}^{\infty} \frac{\overline{d_n}}{n^{1+u}} +o_{u,D}(T).
 \label{eq:mixed-zeta-D-prelim}
\end{align}
Moreover,
\begin{align}
 \int_{T_0}^{T}|D(s_t)|^2dt
 &= T\sum_{n=1}^{\infty}\frac{|d_n|^2}{n} + O_D(1).
 \label{eq:D-square-prelim}
\end{align}
\end{lemma}

\begin{proof}
We first prove \eqref{eq:mixed-zeta-D-prelim}. 
Put $Y:=T^{3/4}$ and choose $N=\lceil C_uT\rceil$, where $C_u>0$ is a sufficiently large constant depending only on $u$. 
We split the integral as $\int_{T_0}^{T} = \int_{T_0}^{Y} + \int_Y^T$.

By the standard second-moment estimate for the Riemann zeta-function on a fixed vertical line,
\[
 \int_{T_0}^{Y} \left|\zeta(s_t+u)\right|^2 dt \ll_u Y\log(Y+2).
\]
Together with \eqref{eq:D-uniform-bound} and the Cauchy--Schwarz inequality, this gives
\begin{align}
 \int_{T_0}^{Y} \left| \zeta(s_t+u)\overline{D(s_t)} \right| dt
 &\ll_{u,D}
 Y^{1/2}
 \left( \int_{T_0}^{Y}|\zeta(s_t+u)|^2dt \right)^{1/2} \notag\\
 &\ll_{u,D}Y(\log Y)^{1/2}
 = o(T).
 \label{eq:mixed-small-range}
\end{align}

We now consider $Y\leq t\leq T$.
By the classical truncation formula \cite[Theorem~4.11, equation~(4.11.1)]{Titchmarsh}, for any fixed $C>1$, uniformly when $\left|t+\Im u\right|<2\pi N/C$,
we have
\begin{align}
 \zeta(s_t+u)
 &= \sum_{m\leq N}\frac{1}{m^{s_t+u}} + \frac{N^{1-s_t-u}}{s_t+u-1} + O_u\left(N^{-1/2-\Re u}\right).
 \label{eq:zeta-truncated-for-mixed}
\end{align}
Since $t\geq Y=T^{3/4}$ and $N\asymp_u T$, we have
\[
 \frac{N^{1/2-\Re u}}{|s_t+u-1|}
 \ll_u\frac{T^{1/2}}{Y}
 \ll T^{-1/4}.
\]
Consequently,
\begin{equation}
 \zeta(s_t+u) = P_u(t)+O_u(T^{-1/4}),
 \qquad
 P_u(t):= \sum_{m\leq N}\frac{1}{m^{1/2+u+it}},
 \label{eq:zeta-polynomial-for-mixed}
\end{equation}
uniformly for $Y\leq t\leq T$. 
By \eqref{eq:D-uniform-bound}, the contribution of the error term in \eqref{eq:zeta-polynomial-for-mixed} is
\[
 \ll_{u,D}T^{-1/4}(T-Y)
 \ll_{u,D}T^{3/4}
 =o(T).
\]

Next, define $D_N(s):=\sum_{n\leq N} d_n n^{-s}$.
Since $d_n\ll n^{-1}$, we have uniformly in $t$
\begin{align}
 |D(s_t)-D_N(s_t)|
 &\leq \sum_{n>N}\frac{|d_n|}{\sqrt n}
 \ll_D N^{-1/2}.
 \label{eq:D-tail-uniform}
\end{align}
The Montgomery--Vaughan mean-value theorem also gives
\begin{align*}
 \int_Y^T|P_u(t)|^2dt
 &\ll (T+N)\sum_{m\leq N}\frac1{m^{1+2\Re u}}
 \ll_u T\log T.
\end{align*}
It follows from Cauchy--Schwarz and \eqref{eq:D-tail-uniform} that
\begin{align*}
 \int_Y^T |P_u(t)| |D(s_t)-D_N(s_t)|\,dt
 &\ll_D N^{-1/2}T^{1/2} \left(\int_Y^T|P_u(t)|^2dt\right)^{1/2}\\
 &\ll_{u,D}T^{1/2}(\log T)^{1/2}
 = o(T).
\end{align*}
We have therefore reduced the problem to the mean value $\int_Y^T P_u(t)\overline{D_N(s_t)}\,dt$.

Apply \eqref{eq:MV-bilinear} with $a_n=n^{-1/2-u},~b_n=d_nn^{-1/2}$.
The diagonal term is $(T-Y)\sum_{n\leq N}\overline{d_n} n^{-1-u}$.
Moreover,
\[
 \sum_{n\leq N}n|a_n|^2 = \sum_{n\leq N}n^{-2\Re u} \ll_u N 
 \qquad\text{and}\qquad
 \sum_{n\leq N}n|b_n|^2 = \sum_{n\leq N}|d_n|^2 \ll_D 1.
\]
Hence \eqref{eq:MV-bilinear} gives
\begin{align}
 \int_Y^T P_u(t)\overline{D_N(s_t)}\,dt
 &= (T-Y) \sum_{n\leq N} \frac{\overline{d_n}}{n^{1+u}} + O_{u,D}(N^{1/2}).
 \label{eq:mixed-MV-application}
\end{align}

The series $\sum_{n=1}^{\infty} \overline{d_n}/n^{1+u}$ is absolutely convergent. 
In fact,
\[
 \sum_{n>N} \left| \frac{\overline{d_n}}{n^{1+u}} \right|
 \ll_{u,D} \sum_{n>N}n^{-2-\Re u}
 \ll_{u,D} N^{-1-\Re u}.
\]
Since $N\asymp T$ and $Y=o(T)$, it follows that
\begin{align*}
 (T-Y) \sum_{n\leq N}\frac{\overline{d_n}}{n^{1+u}}
 &= T\sum_{n=1}^{\infty} \frac{\overline{d_n}}{n^{1+u}} + o_{u,D}(T).
\end{align*}
Combining this with \eqref{eq:mixed-small-range}--\eqref{eq:mixed-MV-application} proves \eqref{eq:mixed-zeta-D-prelim}.

We next prove \eqref{eq:D-square-prelim}.  
Applying \eqref{eq:MV-bilinear} with $a_n=b_n=d_n/\sqrt n$, we obtain
\begin{align}
 \int_{T_0}^{T} |D_M(s_t)|^2 dt
 &= (T-T_0)\sum_{n\leq M}\frac{|d_n|^2}{n} + O\left(\sum_{n\leq M}|d_n|^2\right).
 \label{eq:D-N-square}
\end{align}
for $M\ge1$.
Because $d_n\ll n^{-1}$, both $\sum_{n=1}^{\infty}|d_n|^2/n$ and $\sum_{n=1}^{\infty}|d_n|^2$
converge. 
Furthermore, $D_M(s_t)$ converges uniformly to $D(s_t)$ for $t\in\R$. 
We may therefore let $M\to\infty$ in \eqref{eq:D-N-square}. 
This yields
\[
 \int_{T_0}^{T} |D(s_t)|^2 dt
 = (T-T_0)\sum_{n=1}^{\infty}\frac{|d_n|^2}{n} + O_D(1),
\]
which is equivalent to \eqref{eq:D-square-prelim}, since $T_0$ is fixed.
\end{proof}

We shall also need the corresponding mixed mean involving the derivative of the zeta-function.  

\begin{corollary}
\label{cor:A-D-mixed-prelim}
We have
\begin{align}
 \int_2^T (-\zeta'(s_t)+\gamma\zeta(s_t))\overline{D(s_t)}\,dt
 &= T\sum_{n=1}^{\infty} \frac{(\log n+\gamma)\overline{d_n}}{n} + o_D(T).
\end{align}
\end{corollary}

\begin{proof}
Put again $Y=T^{3/4},~ N=\lceil CT\rceil$,
where $C>0$ is sufficiently large. 
Differentiating the truncated formula \eqref{eq:zeta-truncated-for-mixed}, or equivalently applying Cauchy's integral formula to it, gives, uniformly for $Y\leq t\leq T$,
\begin{align}
 -\zeta'(s_t)+\gamma\zeta(s_t)
 &= \sum_{m\leq N} \frac{\log m+\gamma}{m^{1/2+it}} + O\left(T^{-1/4}\log T\right).
 \label{eq:A-truncated}
\end{align}
Indeed, differentiation of $N^{1-s}/(s-1)$ produces terms bounded by $O( N^{1/2}\log N/t + N^{1/2}/t^2)=O(T^{-1/4}\log T)$ for $t\geq Y$.

Corollary~\ref{cor:mean-square-A} implies
\[
 \int_2^Y|-\zeta'(s_t)+\gamma\zeta(s_t)|^2dt\ll Y\log^3(Y+2).
\]
Since $D(s_t)$ is bounded on the critical line, Cauchy--Schwarz gives
\[
 \int_2^Y |(-\zeta'(s_t)+\gamma\zeta(s_t))D(s_t)| \,dt
 \ll_D Y\log^{3/2}(Y+2)
 =o(T).
\]
The error term in \eqref{eq:A-truncated} contributes at most $O_D\left(T^{3/4}\log T\right)=o(T)$ on the interval $[Y,T]$.

Define
\[
 P^\sharp(t)
 := \sum_{m\leq N} \frac{\log m+\gamma}{m^{1/2+it}}
 \qquad\text{and}\qquad
 D_N(s) := \sum_{n\leq N}\frac{d_n}{n^s}.
\]
As in the proof of Lemma~\ref{lem:mixed-Dirichlet-prelim}, the tail $D(s_t)-D_N(s_t)$ may be removed. 
Indeed,
\[
 \int_Y^T|P^\sharp(t)|^2dt
 \ll (T+N) \sum_{m\leq N}\frac{(\log m+\gamma)^2}{m}
 \ll T\log^3T,
\]
and hence
\begin{align*}
 \int_Y^T |P^\sharp(t)| |D(s_t)-D_N(s_t)| \,dt
 &\ll_D N^{-1/2}T^{1/2} \left(\int_Y^T|P^\sharp(t)|^2dt\right)^{1/2}\\
 &\ll_D T^{1/2}\log^{3/2}T
 =o(T).
\end{align*}

Finally, apply \eqref{eq:MV-bilinear} with $a_n=(\log n+\gamma)/\sqrt n,~ b_n=d_n/\sqrt n$.
We obtain
\begin{align}
 \int_Y^T P^\sharp(t)\overline{D_N(s_t)} \,dt
 &= (T-Y) \sum_{n\leq N} \frac{(\log n+\gamma)\overline{d_n}}{n} + O_D\left(N^{1/2}\log N\right).
 \label{eq:A-D-MV}
\end{align}
Here we used
\[
 \sum_{n\leq N}n|a_n|^2
 = \sum_{n\leq N}(\log n+\gamma)^2
 \ll N\log^2N
\]
and
\[
 \sum_{n\leq N}n|b_n|^2 = \sum_{n\leq N}|d_n|^2 \ll_D1.
\]

The series $\sum_{n=1}^{\infty} (\log n+\gamma)\overline{d_n}/n$ is absolutely convergent, and its tail satisfies
\[
 \sum_{n>N} \left| \frac{(\log n+\gamma)\overline{d_n}}{n} \right|
 \ll_D\frac{\log N}{N}.
\]
Therefore the diagonal term in \eqref{eq:A-D-MV} equals
\[
 T\sum_{n=1}^{\infty} \frac{(\log n+\gamma)\overline{d_n}}{n} + o_D(T).
\]
\end{proof}

\section{The diagonal coefficient and multiple zeta values}
\label{sec:diagonal-coefficient}

Here, we prove Theorem \ref{thm:intro-diagonal-MZV}.

\begin{proof}
Expanding the square in \eqref{eq:intro-diagonal-coefficient} and using $\overline{m^{-s_1}}=m^{-\overline{s_1}}$, we obtain
\begin{align*}
 \zeta_2^{[2]}(s_1,\rho)
 &= \sum_{n=1}^{\infty} \frac1{n^{\rho}} \left(\sum_{m=1}^{n-1}\frac1{m^{s_1}}\right) \overline{\left(\sum_{k=1}^{n-1}\frac1{k^{s_1}}\right)}
 = \sum_{n=1}^{\infty} \frac1{n^{\rho}} \sum_{m=1}^{n-1}\frac1{m^{s_1}} \sum_{k=1}^{n-1}\frac1{k^{\overline{s_1}}}\\
 &= \sum_{1\leq m,k<n} \frac1{m^{s_1}k^{\overline{s_1}}n^{\rho}}.
\end{align*}
We divide the range of summation into the three disjoint parts
\[
 m<k<n,\qquad k<m<n,\qquad m=k<n.
\]
The first part gives
\[
 \sum_{1\leq m<k<n} \frac1{m^{s_1}k^{\overline{s_1}}n^{\rho}}
 = \zeta_3\left(s_1,\overline{s_1},\rho\right),
\]
while the second, upon interchanging the roles of $m$ and $k$, gives
\[
 \sum_{1\leq k<m<n} \frac1{m^{s_1}k^{\overline{s_1}}n^{\rho}}
 = \zeta_3\left(\overline{s_1},s_1,\rho\right).
\]
Finally, the terms with $m=k$ contribute, since $s_1+\overline{s_1}=2\sigma_1$,
\[
 \sum_{1\leq m<n} \frac1{m^{s_1+\overline{s_1}}n^{\rho}}
 = \zeta_2\left(2\sigma_1,\rho\right).
\]
Adding the three contributions yields \eqref{eq:intro-diagonal-MZV}.  
\end{proof}

\section{Mean-square formulas on the boundary}
\label{sec:boundary-mean-square}

We now apply the decompositions and mean-value formulas established in the preceding section.  

\subsection{The case \texorpdfstring{$\zeta_2(1+a,\frac12+it)$}{zeta2(1+a,1/2+it)}}
\label{subsec:plus-family}

Here, we prove the first equation in Theorem \ref{thm:intro-boundary-main}.

\begin{proof}
By the Euler--Maclaurin decomposition obtained in Proposition~\ref{prop:boundary-decompositions}, we have
\begin{align}
 \zeta_2(1+a,s)
 &= \zeta(1+a)\zeta(s) - \frac1a\zeta(s+a) + B_a^{+}(s).
 \label{eq:plus-decomposition-boundary}
\end{align}
Note that $B_a^{+}(s)$ is absolutely convergent on $\Re s=1/2$, and Lemma~\ref{lem:mixed-Dirichlet-prelim} is applicable to it.

For fixed $u,v$, write
\[
 M_{u,v}(T) := \int_{T_a}^{T} \zeta(s_t+u)\, \overline{\zeta(s_t+v)} \,dt
\]
and
\[
 L_u(T) := \int_{T_a}^{T} \zeta(s_t+u)\, \overline{B_a^{+}(s_t)} \,dt.
\]
On inserting \eqref{eq:plus-decomposition-boundary} and expanding the square, we obtain
\begin{align}
 \int_{T_a}^{T} \left|\zeta_2(1+a,s_t)\right|^2 dt
 &= |\zeta(1+a)|^2M_{0,0}(T) + \frac1{|a|^2}M_{a,a}(T)
 - 2\Re\left\{ \frac{\zeta(1+a)}{\overline a}M_{0,a}(T) \right\} \notag\\
 &\quad
 +2\Re\left\{\zeta(1+a)L_0(T)\right\}-2\Re\left\{\frac1aL_a(T)\right\}+\int_{T_a}^{T}|B_a^{+}(s_t)|^2dt.
 \label{eq:plus-expanded-square}
\end{align}

We evaluate the terms on the right-hand side separately. 
The classical second-moment formula gives
\begin{align}
 M_{0,0}(T)
 &= T\log\frac{T}{2\pi} + (2\gamma-1)T + o(T).
 \label{eq:plus-M00}
\end{align}
Corollary~\ref{cor:fixed-shifted-second-moment-positive} gives
\begin{align}
 M_{0,a}(T) &= T\zeta(1+\overline a)+o_a(T), 
 \label{eq:plus-M0a}\\
 M_{a,a}(T) &= T\zeta(1+a+\overline a)+o_a(T).
 \label{eq:plus-Maa}
\end{align}
Notice that $1+a+\overline a=1+2\Re a>1$.

Lemma~\ref{lem:mixed-Dirichlet-prelim}, applied to $D(s)=B_a^{+}(s)$, yields
\begin{align}
 L_0(T) &= T\sum_{n=1}^{\infty} \frac{\overline{r_a^{+}(n)}}{n} + o_a(T), 
 \label{eq:plus-L0}\\
 L_a(T) &= T\sum_{n=1}^{\infty} \frac{\overline{r_a^{+}(n)}}{n^{1+a}} + o_a(T),  
 \label{eq:plus-La}
\end{align}
and
\begin{align}
 \int_{T_a}^{T} |B_a^{+}(s_t)|^2 dt
 &= T\sum_{n=1}^{\infty} \frac{|r_a^{+}(n)|^2}{n} + o_a(T).
 \label{eq:plus-B-square}
\end{align}

Substitution of \eqref{eq:plus-M00}--\eqref{eq:plus-B-square} into \eqref{eq:plus-expanded-square} gives
\begin{align}
 &\int_{T_a}^{T} \left|\zeta_2(1+a,s_t)\right|^2 dt \notag\\
 &=
 |\zeta(1+a)|^2T\log\frac{T}{2\pi}
 +T\Biggl[ |\zeta(1+a)|^2(2\gamma-1)+\frac{\zeta(1+2\Re a)}{|a|^2}
 -2\Re\left\{ \frac{|\zeta(1+a)|^2}{\overline a} \right\} \notag\\
 &\qquad\quad
 +2\Re\left\{ \zeta(1+a)\sum_{n=1}^{\infty} \frac{\overline{r_a^{+}(n)}}{n} \right\}
 -2\Re\left\{ \frac1a\sum_{n=1}^{\infty} \frac{\overline{r_a^{+}(n)}}{n^{1+a}} \right\}
 +\sum_{n=1}^{\infty} \frac{|r_a^{+}(n)|^2}{n} \Biggr]
 + o_a(T).
 \label{eq:plus-constant-expanded}
\end{align}

It remains to rewrite the expression in brackets in terms of the original generalized harmonic sums. 
From \eqref{eq:H-plus-AET},
\begin{align*}
 \left|H_{n-1}^{(1+a)}\right|^2-|\zeta(1+a)|^2
 &=
 \frac1{|a|^2}n^{-a-\overline a}
 -2\Re\left\{ \frac{\zeta(1+a)}{\overline a}n^{-\overline a} \right\}
 \notag\\
 &\quad
 +2\Re\left\{ \zeta(1+a)\overline{r_a^{+}(n)} \right\}
 -2\Re\left\{ \frac{\overline{r_a^{+}(n)}}{a n^a} \right\}
 +|r_a^{+}(n)|^2.
\end{align*}
After dividing by $n$ and summing over $n$, the first two terms on the right give
\[
 \frac{\zeta(1+a+\overline a)}{|a|^2}
 -2\Re\left\{ \frac{|\zeta(1+a)|^2}{\overline a} \right\}.
\]
Thus the entire expression in brackets in \eqref{eq:plus-constant-expanded} is precisely $C_{+}(a)$.
\end{proof}

\begin{remark}\label{rem:plus-renormalized-diagonal}
The convergence of the series in $C_+(a)$ follows directly from
\[
 H_{n-1}^{(1+a)}-\zeta(1+a) = -\frac1a n^{-a} + O_a\left(n^{-1-\Re a}\right).
\]
Indeed, putting $\delta_n:=H_{n-1}^{(1+a)}-\zeta(1+a)$, we have
\begin{align*}
 \left|H_{n-1}^{(1+a)}\right|^2 - |\zeta(1+a)|^2
 &=\left|\zeta(1+a)+\delta_n\right|^2-|\zeta(1+a)|^2\\
 &= 2\operatorname{Re} \left\{ \overline{\zeta(1+a)}\,\delta_n \right\} + |\delta_n|^2\\
 &\ll_a n^{-\Re a}.
\end{align*}
Consequently,
\[
 \sum_{n=1}^{\infty} \frac{ |H_{n-1}^{(1+a)}|^2-|\zeta(1+a)|^2 }{n}
 \ll_a \sum_{n=1}^{\infty}\frac1{n^{1+\Re a}}
 < \infty.
\]
\end{remark}

\begin{remark}\label{Remark4.2}
The two terms in $C_+(a)$ have different origins. 
The first term $|\zeta(1+a)|^2(2\gamma-1)$ is inherited from the classical mean-square formula for the Riemann zeta-function. 
The second term,
\[
 \sum_{n=1}^{\infty} \frac{|H_{n-1}^{(1+a)}|^2-|\zeta(1+a)|^2}{n},
\]
is the harmonic finite part of the diagonal Dirichlet series $\zeta_2^{[2]}(s_1,\rho)$ at $(s_1,\rho)=(1+a,1)$, which corresponds to the boundary point $(s_1,\sigma_2)=(1+a,1/2)$ of the mean-square problem.
Indeed, it is obtained from the truncated diagonal sum by subtracting its divergent part $|\zeta(1+a)|^2H_N$ and then letting $N\to\infty$. 
When $a=k-1$ with $k\geq2$, this harmonic finite part becomes $2\zeta^*(k,k,1)+\zeta^*(2k,1)$.
\end{remark}

\begin{remark}\label{Remark4.3}
The finite part above is closely related to an Euler-type constant of $\zeta_2^{[2]}(1+a,\rho) = \zeta_3(1+a,1+\overline{a},\rho)+\zeta_3(1+\overline{a},1+a,\rho)+\zeta_2(2+2\Re a,\rho)$. 
Regard $\zeta_2^{[2]}(1+a,\rho)$ as a function of the complex variable $\rho$. 
Then it has a simple pole at $\rho=1$ with residue $|\zeta(1+a)|^2$, and
\begin{align}
 \zeta_2^{[2]}(1+a,\rho)
 &= \frac{|\zeta(1+a)|^2}{\rho-1} + \gamma_+(a) + O_a\bigl(|\rho-1|\bigr), 
 \label{eq:gamma-plus-Laurent}
\end{align}
where
\begin{align}
 \gamma_+(a)
 &:= \lim_{N\to\infty} \left\{ \sum_{n=1}^{N} \frac{|H_{n-1}^{(1+a)}|^2}{n}-|\zeta(1+a)|^2\log N \right\}. 
 \label{eq:def-gamma-plus}
\end{align}

Indeed, as observed above, the Dirichlet series
\[
 R_a(\rho) := \sum_{n=1}^{\infty} \frac{|H_{n-1}^{(1+a)}|^2-|\zeta(1+a)|^2}{n^\rho}
\]
converges absolutely and defines a holomorphic function in a neighborhood of $\rho=1$. 
Hence $R_a(\rho)=R_a(1)+O_a(|\rho-1|)$.
Since $\zeta_2^{[2]}(1+a,\rho)=|\zeta(1+a)|^2\zeta(\rho)+R_a(\rho)$, the Laurent expansion of the Riemann zeta-function at $\rho=1$ gives
\begin{align*}
 \zeta_2^{[2]}(1+a,\rho)
 &= \frac{|\zeta(1+a)|^2}{\rho-1} + \gamma|\zeta(1+a)|^2 + R_a(1) + O_a\bigl(|\rho-1|\bigr).
\end{align*}
On the other hand, using $\lim_{N\to\infty}(\sum_{n\leq N}n^{-1}-\log N) = \gamma$, we obtain
\begin{align*}
 \gamma_+(a)
 &= \gamma|\zeta(1+a)|^2 + R_a(1).
\end{align*}
Thus the regularized limit \eqref{eq:def-gamma-plus} coincides with the constant term in the Laurent expansion \eqref{eq:gamma-plus-Laurent}.

Note that the normalization in \eqref{eq:def-gamma-plus} subtracts $|\zeta(1+a)|^2\log N$, whereas the harmonic finite part in Remark~\ref{Remark4.2} subtracts $|\zeta(1+a)|^2H_N$.
The latter finite part is therefore $R_a(1)=\gamma_+(a)-\gamma|\zeta(1+a)|^2$, and consequently $C_{+}(a)=\gamma_+(a)+(\gamma-1)|\zeta(1+a)|^2$.
\end{remark}

\subsection{The case \texorpdfstring{$\zeta_2(1+ib,\frac12+it)$}{zeta2(1+ib,1/2+it)}}
\label{subsec:imaginary-family}

Throughout this subsection, let $b\in\mathbb R\setminus\{0\}$ be fixed. 
The case $b=0$ will be treated separately in the next subsection.

\begin{proof}[Proof of Theorem~\ref{thm:intro-boundary-main}, second formula]
Recall from \eqref{eq:decomposition-imaginary} that
\begin{equation*}
 \zeta_2(1+ib,s_t) = \zeta(1+ib)\zeta(s_t) + \frac i{b}\zeta(s_t+ib) + B_b(s_t).
\end{equation*}
Expanding the square and integrating, we obtain
\begin{align*}
 \int_2^T|\zeta_2(1+ib,s_t)|^2dt
 &= |\zeta(1+ib)|^2\int_2^T|\zeta(s_t)|^2dt
 +\left|\frac{i}{b}\right|^2\int_2^T|\zeta(s_t+ib)|^2dt
 +\int_2^T|B_b(s_t)|^2dt
 \notag\\
 &\quad
 +2\operatorname{Re}\left\{\zeta(1+ib)\overline{\frac{i}{b}}\int_2^T \zeta(s_t)\overline{\zeta(s_t+ib)}\,dt\right\}
 \notag\\
 &\quad
 +2\operatorname{Re}\left\{\zeta(1+ib)\int_2^T \zeta(s_t)\overline{B_b(s_t)}\,dt\right\}
 \notag\\
 &\quad
 +2\operatorname{Re}\left\{\frac{i}{b}\int_2^T \zeta(s_t+ib)\overline{B_b(s_t)}\,dt\right\}.
\end{align*}

We now evaluate the six terms on the right-hand side. 
By \eqref{eq:classical-zeta-mean-square} and \eqref{eq:pure-imaginary-shift-square},
\begin{align*}
 &\int_2^T|\zeta(s_t)|^2dt = TX+(2\gamma-1)T+o(T),\\
 &\int_2^T|\zeta(s_t+ib)|^2dt = TX+(2\gamma-1)T+o_b(T),
\end{align*}
where $X=\log(T/(2\pi))$. 
Corollary~\ref{cor:pure-imaginary-shift} gives
\begin{align}
 \int_2^T \zeta(s_t)\overline{\zeta(s_t+ib)}\,dt
 &=T\zeta(1-ib) + \frac{\zeta(1+ib)}{1+ib}Te^{ibX} + o_b(T).
 \label{eq:imaginary-family-zeta-cross}
\end{align}

Next apply Lemma~\ref{lem:mixed-Dirichlet-prelim} with $D=B_b$. 
First, taking $u=0$, we obtain
\begin{align*}
 \int_2^T \zeta(s_t)\overline{B_b(s_t)}\,dt
 &=T\sum_{n=1}^{\infty}\frac{\overline{r_b(n)}}{n} + o_b(T) \notag\\
 &=T\overline{B_b(1)}+o_b(T).
\end{align*}
Taking $u=ib$ instead gives
\begin{align*}
 \int_2^T \zeta(s_t+ib)\overline{B_b(s_t)}\,dt
 &=T\sum_{n=1}^{\infty} \frac{~\overline{r_b(n)}~}{n^{1+ib}} + o_b(T) \notag\\
 &=T\overline{B_b(1-ib)}+o_b(T).
\end{align*}
Finally, \eqref{eq:D-square-prelim} gives
\begin{align*}
 \int_2^T|B_b(s_t)|^2dt
 &=T\sum_{n=1}^{\infty}\frac{|r_b(n)|^2}{n}+O_b(1).
\end{align*}

It remains to collect these formulas. 
The first term of the right-hand side of \eqref{eq:imaginary-family-zeta-cross} vanishes, since
\begin{align*}
 2\operatorname{Re}\left\{ \zeta(1+ib)\overline{\frac{i}{b}}\, T\zeta(1-ib) \right\}
 = 2T|\zeta(1+ib)|^2 \operatorname{Re}\left(-\frac{i}{b}\right)
 = 0.
\end{align*}
Its remaining part is
\begin{align*}
 &2\operatorname{Re}\left\{\zeta(1+ib)\overline{\frac{i}{b}}\,\frac{\zeta(1+ib)}{1+ib}Te^{ibX}\right\}
 = -2T\operatorname{Re}\left\{\frac{i\zeta(1+ib)^2}{b(1+ib)}e^{ibX}\right\}
 = T\mathcal E_b(X).
\end{align*}

We therefore obtain
\begin{align*}
 &\int_2^T|\zeta_2(1+ib,s_t)|^2dt\\
 &=
 \left(|\zeta(1+ib)|^2+b^{-2}\right)TX+T\mathcal E_b(X)\\
 &\quad+
 T\Biggl[
  (2\gamma-1)\left(|\zeta(1+ib)|^2+b^{-2}\right)
  +2\operatorname{Re}\left\{
    \zeta(1+ib)\overline{B_b(1)}
    +\frac{i}{b}\overline{B_b(1-ib)}
   \right\}
  +\sum_{n=1}^{\infty}\frac{|r_b(n)|^2}{n}
 \Biggr]\\
 &\quad+o_b(T).
\end{align*}

The constant expression in square brackets has a more compact form. 
Indeed, by \eqref{eq:H-imaginary-AET}, we have
\begin{align*}
 \left|H_{n-1}^{(1+ib)}\right|^2 - \left|\zeta(1+ib)+\frac{i}{b}n^{-ib}\right|^2
 = 2\operatorname{Re}\left\{\left(\zeta(1+ib)+\frac{i}{b}n^{-ib}\right)\overline{r_b(n)}\right\} + |r_b(n)|^2.
\end{align*}
Dividing this identity by $n$ and then summing over $n\geq1$, we obtain
\begin{align}
 \mathcal D_b
 &= (2\gamma-1)\left(|\zeta(1+ib)|^2+b^{-2}\right)
 +\sum_{n=1}^{\infty}
 \frac{|H_{n-1}^{(1+ib)}|^2-\left|\zeta(1+ib)+\frac{i}{b}n^{-ib}\right|^2}{n}.
 \label{eq:D-b-compact}
\end{align}
Since $r_b(n)\ll_b n^{-1}$ by \eqref{eq:r-imaginary-AET}, the summand in \eqref{eq:D-b-compact} is $O_b(n^{-2})$. 
Hence the series converges absolutely. 
This completes the proof.
\end{proof}

\begin{remark}\label{Remark4.5}
The convergent series in \eqref{eq:D-b-compact} also has an interpretation in terms of an Euler-type constant. 
Define
\begin{align*}
 \gamma_b
 &:= \gamma\left(|\zeta(1+ib)|^2+b^{-2}\right)
 +\sum_{n=1}^{\infty} \frac{|H_{n-1}^{(1+ib)}|^2-\left|\zeta(1+ib)+\frac{i}{b}n^{-ib}\right|^2}{n}.
\end{align*}
Then, regarding
$\zeta_2^{[2]}(1+ib,\rho)=\zeta_3(1+ib,1-ib,\rho)+\zeta_3(1-ib,1+ib,\rho)+\zeta_2(2,\rho)$ as a function of $\rho$, we have
\begin{align}
 \zeta_2^{[2]}(1+ib,\rho)
 &= \frac{|\zeta(1+ib)|^2+b^{-2}}{\rho-1} + \gamma_b + O_b\bigl(|\rho-1|\bigr). 
 \label{eq:gamma-b-Laurent}
\end{align}

Indeed, if
\[
 R_b(\rho)
 := \sum_{n=1}^{\infty} \frac{|H_{n-1}^{(1+ib)}|^2-\left|\zeta(1+ib)+\frac{i}{b}n^{-ib}\right|^2}{n^\rho},
\]
then $R_b(\rho)$ is holomorphic in a neighborhood of $\rho=1$, and
\begin{align*}
 \zeta_2^{[2]}(1+ib,\rho)
 &= \left(|\zeta(1+ib)|^2+b^{-2}\right)\zeta(\rho) - \frac{i}{b}\zeta(1+ib)\zeta(\rho-ib) + \frac{i}{b}\overline{\zeta(1+ib)}\zeta(\rho+ib) + R_b(\rho).
\end{align*}
At $\rho=1$, the second and third terms on the right-hand side cancel.
This proves \eqref{eq:gamma-b-Laurent}.

Equivalently,
\begin{align*}
 \gamma_b
 = \lim_{N\to\infty}
 \left\{\sum_{n=1}^{N}\frac{|H_{n-1}^{(1+ib)}|^2}{n}-\left(|\zeta(1+ib)|^2+b^{-2}\right)\log N+\frac{2}{b^2}\operatorname{Re}\left(\zeta(1+ib)N^{ib}\right)\right\}.
\end{align*}
Thus $\gamma_b$ is the finite part obtained after removing both the logarithmic term and the bounded oscillatory term from the partial sums of the diagonal series. 
In terms of this Euler-type constant, \eqref{eq:D-b-compact} becomes $\mathcal D_b=\gamma_b+(\gamma-1)\left(|\zeta(1+ib)|^2+b^{-2}\right)$.
\end{remark}

\subsection{The case \texorpdfstring{$\zeta_2(1,\frac12+it)$}{zeta2(1,1/2+it)}}
\label{subsec:corner-family}

Here, we prove the third equation in Theorem \ref{thm:intro-boundary-main}.

\begin{proof}
\eqref{eq:decomposition-zero} gives
\begin{align*}
 \zeta_2(1,s) &= -\zeta'(s)+\gamma\zeta(s)+B_0(s).
\end{align*}
It follows that
\begin{align}
 \int_2^T|\zeta_2(1,s_t)|^2dt
 &= \int_2^T|-\zeta'(s_t)+\gamma\zeta(s_t)|^2dt
 +2\Re\int_2^T (-\zeta'(s_t)+\gamma\zeta(s_t))\overline{B_0(s_t)}\,dt \notag\\
 &\quad +\int_2^T|B_0(s_t)|^2dt.
 \label{eq:corner-square-expansion}
\end{align}

By Corollary~\ref{cor:mean-square-A}, we have
\begin{align}
 \int_2^T|-\zeta'(s_t)+\gamma\zeta(s_t)|^2dt
 &= T P_{-\zeta'+\gamma\zeta}\left(\log\frac{T}{2\pi}\right)+o(T).
 \label{eq:corner-A-square}
\end{align}

The mixed mean-value formula in Corollary~\ref{cor:A-D-mixed-prelim}, applied to $D(s)=B_0(s)$, gives
\begin{align}
 \int_2^T (-\zeta'(s_t)+\gamma\zeta(s_t))\overline{B_0(s_t)}\,dt
 &= T\sum_{n=1}^{\infty} \frac{(\log n+\gamma)r_0(n)}{n} + o(T).
 \label{eq:corner-A-B-mixed}
\end{align}
Here $r_0(n)$ is real. 
Similarly, Lemma~\ref{lem:mixed-Dirichlet-prelim} gives
\begin{align}
 \int_2^T|B_0(s_t)|^2dt &= T\sum_{n=1}^{\infty} \frac{r_0(n)^2}{n} + o(T).
 \label{eq:corner-B-square}
\end{align}

Substituting \eqref{eq:corner-A-square}, \eqref{eq:corner-A-B-mixed}, and \eqref{eq:corner-B-square} into \eqref{eq:corner-square-expansion}, we obtain
\begin{align}
 \int_2^T|\zeta_2(1,s_t)|^2dt
 &= T P_{-\zeta'+\gamma\zeta}(X)+T\sum_{n=1}^{\infty} \frac{2(\log n+\gamma)r_0(n)+r_0(n)^2}{n}+o(T). 
 \label{eq:corner-before-R0}
\end{align}
By \eqref{eq:H-zero-AET}, we have
\begin{align}
 \sum_{n=1}^{\infty} \frac{2(\log n+\gamma)r_0(n)+r_0(n)^2}{n}
 = \sum_{n=1}^{\infty} \frac{H_{n-1}^2-(\log n+\gamma)^2}{n}.
 \label{eq:def-R-zero}
\end{align}
The series in \eqref{eq:def-R-zero} is convergent, since $r_0(n)=O(n^{-1})$ by \eqref{eq:r-zero-AET}.

For $N\geq1$, we introduce the truncated multiple zeta sums
\[
 \zeta_N(k_1,\ldots,k_r)
 := \sum_{1\leq n_1<\cdots<n_r\leq N} \frac1{n_1^{k_1}\cdots n_r^{k_r}}.
\]
Expanding $H_{n-1}^2$ gives
\begin{align}
 \sum_{n=1}^{N}\frac{H_{n-1}^2}{n} &= 2\zeta_N(1,1,1)+\zeta_N(2,1).
 \label{eq:corner-harmonic-multiple}
\end{align}
On the other hand, the finite harmonic product identity gives
\begin{align}
 H_N^3 &= 6\zeta_N(1,1,1) + 3\zeta_N(2,1) + 3\zeta_N(1,2) + \zeta_N(3).
 \label{eq:finite-cubic-harmonic-product}
\end{align}
Combining \eqref{eq:corner-harmonic-multiple} and \eqref{eq:finite-cubic-harmonic-product}, we find
\begin{align}
 \sum_{n=1}^{N}\frac{H_{n-1}^2}{n}
 &= \frac13H_N^3 - \zeta_N(1,2) - \frac13\zeta_N(3).
 \label{eq:corner-harmonic-square-identity}
\end{align}

As $N\to\infty$,
\begin{align*}
 &H_N = \log N+\gamma+O\left(\frac1N\right),\\
 &\zeta_N(1,2) \longrightarrow \zeta_2(1,2)=\zeta(3),\\
 &\zeta_N(3) \longrightarrow \zeta(3).
\end{align*}
Therefore, we obtain
\begin{align}
 \sum_{n=1}^{N}\frac{H_{n-1}^2}{n}
 &= \frac13(\log N+\gamma)^3 - \frac43\zeta(3)+o(1).
 \label{eq:corner-H-square-asymptotic}
\end{align}

By the definition of the Stieltjes constants
\begin{align*}
 &\sum_{n=1}^{N}\frac{\log n}{n} = \frac12\log^2 N+\gamma_1+o(1),\\
 &\sum_{n=1}^{N}\frac{(\log n)^2}{n} = \frac13\log^3 N+\gamma_2+o(1),
\end{align*}
we have
\begin{align}
 \sum_{n=1}^{N}\frac{(\log n+\gamma)^2}{n}
 &= \frac13\log^3 N+\gamma \log^2 N+\gamma^2\log N+\gamma^3+2\gamma\gamma_1+\gamma_2+o(1).
 \label{eq:shifted-log-square-asymptotic}
\end{align}
Subtracting \eqref{eq:shifted-log-square-asymptotic} from \eqref{eq:corner-H-square-asymptotic} gives
\begin{align}
 \sum_{n=1}^{\infty} \frac{2(\log n+\gamma)r_0(n)+r_0(n)^2}{n}
 &= -\frac23\gamma^3-2\gamma\gamma_1-\gamma_2-\frac43\zeta(3).
 \label{eq:R-zero-evaluation}
\end{align}

Finally, adding \eqref{eq:R-zero-evaluation} to the constant term of $P_{-\zeta'+\gamma\zeta}(X)$ in \eqref{eq:corner-before-R0}, we obtain $C_0$.
\end{proof}

\begin{remark}\label{rem:corner-regularization}
The identities \eqref{eq:corner-harmonic-multiple} and \eqref{eq:corner-harmonic-square-identity} also give
\begin{align}
 &\lim_{N\to\infty} \left\{2\zeta_N(1,1,1)+\zeta_N(2,1)-\frac13H_N^3\right\}
 = -\frac43\zeta(3).
 \label{eq:corner-regularized-MZV}
\end{align}
Thus the occurrence of $-4\zeta(3)/3$ in $C_0$ is not accidental. 
It is the finite part of the divergent diagonal multiple zeta sum at the corner. 
We shall return to its interpretation in terms of regularized multiple zeta values in Section~\ref{sec:regularized-MZV}.
\end{remark}

\subsection{The case \texorpdfstring{$\zeta_2(1-a,\frac12+a+it)$}{zeta2(1-a,1/2+a+it)}}
\label{subsec:minus-family}

Finally, we prove the last equation in Theorem \ref{thm:intro-boundary-main}.

\begin{proof}
The Euler--Maclaurin decomposition in Proposition~\ref{prop:boundary-decompositions} gives
\begin{align}
 \zeta_2(1-a,s+a) &= \frac1a\zeta(s) + \zeta(1-a)\zeta(s+a) + B_a^{-}(s).  
 \label{eq:minus-decomposition}
\end{align}
Note that $B_a^{-}(s)$ converges absolutely on $\Re s=1/2$.

As before, write
\begin{align*}
 M_{u,v}(T) &:= \int_{T_a}^{T} \zeta(s_t+u)\, \overline{\zeta(s_t+v)}\,dt,\\
 L_u^{-}(T) &:= \int_{T_a}^{T} \zeta(s_t+u)\, \overline{B_a^{-}(s_t)}\,dt.
\end{align*}
Expanding the square in \eqref{eq:minus-decomposition}, we obtain
\begin{align}
 \int_{T_a}^{T} \left|\zeta_2(1-a,s_t+a)\right|^2 dt
 =&
 \frac1{|a|^2}M_{0,0}(T)
 +|\zeta(1-a)|^2M_{a,a}(T)
 +2\Re\left\{\frac{\overline{\zeta(1-a)}}{a}M_{0,a}(T)\right\} \notag\\
 &+2\Re\left\{\frac1aL_0^{-}(T)\right\}+2\Re\left\{\zeta(1-a)L_a^{-}(T)\right\}+\int_{T_a}^{T}|B_a^{-}(s_t)|^2dt.
 \label{eq:minus-expanded-square}
\end{align}

The classical second-moment formula gives
\begin{align*}
 M_{0,0}(T) &= T\log\frac{T}{2\pi} + (2\gamma-1)T + o(T). 
\end{align*}
Corollary \ref{cor:fixed-shifted-second-moment-positive} gives
\begin{align*}
 M_{0,a}(T) &= T\zeta(1+\overline a)+o_a(T), \\
 M_{a,a}(T) &= T\zeta(1+a+\overline a)+o_a(T).
\end{align*}

Since $r_a^{-}(n)\ll_a n^{-1}$, Lemma~\ref{lem:mixed-Dirichlet-prelim} can be applied to $B_a^{-}(s)$. 
It gives
\begin{align*}
 L_0^{-}(T) &= T\sum_{n=1}^{\infty} \frac{\overline{r_a^{-}(n)}}{n} + o_a(T),\\
 L_a^{-}(T) &= T\sum_{n=1}^{\infty} \frac{~\overline{r_a^{-}(n)}~}{n^{1+a}} + o_a(T),
\end{align*}
and
\begin{align*}
 \int_{T_a}^{T}|B_a^{-}(s_t)|^2dt
 &= T\sum_{n=1}^{\infty} \frac{|r_a^{-}(n)|^2}{n} + o_a(T).
\end{align*}

Substituting these formulas into \eqref{eq:minus-expanded-square}, we find
\begin{align}
 \int_{T_a}^{T}  &  \left|\zeta_2(1-a,s_t+a)\right|^2 dt=\frac{T}{|a|^2}\log\frac{T}{2\pi}\notag\\
 &\quad+T\Biggl[\frac{2\gamma-1}{|a|^2}+|\zeta(1-a)|^2\zeta(1+a+\overline a)
 +2\Re\left\{\frac{~\overline{\zeta(1-a)}~}{a}\zeta(1+\overline a)\right\}
 \notag\\
 &\qquad
  +2\Re\left\{\frac1a\sum_{n=1}^{\infty}\frac{\overline{r_a^{-}(n)}}{n}\right\}
 +2\Re\left\{ \zeta(1-a)\sum_{n=1}^{\infty}\frac{~\overline{r_a^{-}(n)}~}{n^{1+a}}\right\}
 +\sum_{n=1}^{\infty}\frac{|r_a^{-}(n)|^2}{n}\Biggr]
 +o_a(T).
 \label{eq:minus-constant-expanded}
\end{align}

We now rewrite the expression in brackets in terms of the original generalized harmonic sums.  
By \eqref{eq:H-minus-AET}, it follows that
\begin{align*}
 |n^{-a}H_{n-1}^{(1-a)}|^2-\frac1{|a|^2}
 &= |\zeta(1-a)|^2n^{-a-\overline a}
 +2\Re\left\{\frac{~\overline{\zeta(1-a)}~}{a}n^{-\overline a}\right\} \notag\\
 &\quad
 +2\Re\left\{\frac{~\overline{r_a^{-}(n)}~}{a}\right\}
 +2\Re\left\{\zeta(1-a)n^{-a}\overline{r_a^{-}(n)}\right\}
 +|r_a^{-}(n)|^2.
\end{align*}
Dividing by $n$ and summing over $n$, we obtain
\begin{align*}
 &\sum_{n=1}^{\infty} \frac{|n^{-a}H_{n-1}^{(1-a)}|^2-|a|^{-2}}{n} \notag\\
 &=
 |\zeta(1-a)|^2\zeta(1+a+\overline a)
 +2\Re\left\{\frac{~\overline{\zeta(1-a)}~}{a}\zeta(1+\overline a)\right\} \notag\\
 &\quad
 +2\Re\left\{\frac1a \sum_{n=1}^{\infty} \frac{\overline{r_a^{-}(n)}}{n}\right\}
 +2\Re\left\{\zeta(1-a) \sum_{n=1}^{\infty} \frac{~\overline{r_a^{-}(n)}~}{n^{1+a}}\right\}
 +\sum_{n=1}^{\infty} \frac{|r_a^{-}(n)|^2}{n}.
\end{align*}
Thus, the expression in brackets in \eqref{eq:minus-constant-expanded} is exactly $C_{-}(a)$.
\end{proof}

\begin{remark}\label{Remark4.7}
The finite part occurring in $C_-(a)$ is likewise closely related to an Euler-type
constant of
$\zeta_2^{[2]}(1-a,\rho+2\Re a) = \zeta_3(1-a,1-\overline{a},\rho+2\Re a)
+\zeta_3(1-\overline{a},1-a,\rho+2\Re a)+\zeta_2(2-2\Re a,\rho+2\Re a)$.
Regard it as a function of the complex variable $\rho$.
Then it has a simple pole at $\rho=1$ with residue $|a|^{-2}$, and
\begin{align}
 \zeta_2^{[2]}(1-a,\rho+2\Re a) &= \frac{1}{|a|^2(\rho-1)} + \gamma_-(a) + O_a\bigl(|\rho-1|\bigr), 
 \label{eq:gamma-minus-Laurent}
\end{align}
where
\begin{align*}
 \gamma_-(a) &:= \lim_{N\to\infty}
 \left\{\sum_{n=1}^{N}\frac{|n^{-a}H_{n-1}^{(1-a)}|^2}{n}-\frac{1}{|a|^2}\log N\right\}.
\end{align*}

Indeed, \eqref{eq:H-minus-AET} gives
$|n^{-a}H_{n-1}^{(1-a)}|^2-|a|^{-2} \ll_a n^{-\min\{\Re a,1\}}$.
Consequently, the Dirichlet series
\[
 R_{-,a}(\rho) := \sum_{n=1}^{\infty} \frac{|n^{-a}H_{n-1}^{(1-a)}|^2-|a|^{-2}}{n^\rho}
\]
converges absolutely and defines a holomorphic function in a neighborhood of $\rho=1$. 
We therefore have
\begin{align*}
 \zeta_2^{[2]}(1-a,\rho+2\Re a)
 &= \frac1{|a|^2}\zeta(\rho)+R_{-,a}(\rho)\\
 &= \frac{1}{|a|^2(\rho-1)} + \frac{\gamma}{|a|^2} + R_{-,a}(1) + O_a\bigl(|\rho-1|\bigr).
\end{align*}
On the other hand, the definition of Euler's constant gives
\begin{align*}
 \gamma_-(a) &= \frac{\gamma}{|a|^2}
 +\sum_{n=1}^{\infty} \frac{|n^{-a}H_{n-1}^{(1-a)}|^2-|a|^{-2}}{n}.
\end{align*}
Thus the regularized limit defining $\gamma_-(a)$ coincides with the constant term
in the Laurent expansion \eqref{eq:gamma-minus-Laurent}.

Note that this normalization subtracts $|a|^{-2}\log N$, whereas subtracting
$|a|^{-2}H_N$ gives the harmonic finite part $R_{-,a}(1)=\gamma_-(a)-\gamma|a|^{-2}$,
which is precisely the series occurring in $C_-(a)$.
Consequently $C_-(a) = \gamma_-(a)+(\gamma-1)|a|^{-2}$.
\end{remark}

\section{Regularized multiple zeta values}
\label{sec:regularized-MZV}

We conclude by interpreting the constants $C_+(k-1)$ and $C_0$.

For a positive integral index $\boldsymbol{k}=(k_1,\ldots,k_r)\in\mathbb N^r$, recall that
\[
 \zeta_N(\boldsymbol{k}) := \sum_{1\leq n_1<\cdots<n_r\leq N} \frac1{n_1^{k_1}\cdots n_r^{k_r}}.
\]
The harmonic regularization of multiple zeta values is described by polynomials $\zeta^{*}(\boldsymbol{k};Y)\in\mathcal Z[Y]$, characterized by
\begin{align}
 \zeta_N(\boldsymbol{k}) &= \zeta^{*}(\boldsymbol{k};H_N)+o(1) \qquad (N\to\infty),
 \label{eq:harmonic-regularization}
\end{align}
where $\mathcal Z$ denotes the algebra generated by convergent multiple zeta values. 
In particular, $\zeta^{*}(1;Y)=Y$.
If $\boldsymbol{k}$ is admissible, then $\zeta^{*}(\boldsymbol{k};Y)=\zeta_r(\boldsymbol{k})$.
We write $\zeta^{*}(\boldsymbol{k}):=\zeta^{*}(\boldsymbol{k};0)$ for the constant term of the harmonic regularization. 
This regularization is compatible with the finite harmonic product relations; see, for example, \cite{IharaKanekoZagier}.

Here, we prove Theorem~\ref{thm:intro-regularized-main}.

\begin{proof}
We first consider the case $k\geq2$. 
Expanding the square of the generalized harmonic number gives the finite identity
\begin{align*}
 \sum_{n=1}^{N}\frac{\left(H_{n-1}^{(k)}\right)^2}{n}
 &= 2\zeta_N(k,k,1)+\zeta_N(2k,1).
\end{align*}

Since $H_{n-1}^{(k)}=\zeta(k)+O_k(n^{1-k})$, the series
\[
 \sum_{n=1}^{\infty} \frac{\left(H_{n-1}^{(k)}\right)^2-\zeta(k)^2}{n}
\]
converges absolutely. 
Therefore,
\begin{align}
 &\lim_{N\to\infty} \left\{2\zeta_N(k,k,1)+\zeta_N(2k,1)-\zeta(k)^2H_N\right\}
 = \sum_{n=1}^{\infty} \frac{\left(H_{n-1}^{(k)}\right)^2-\zeta(k)^2}{n}.
 \label{eq:finite-part-k}
\end{align}
By \eqref{eq:harmonic-regularization}, the left-hand side of \eqref{eq:finite-part-k} is the constant term of $2\zeta^{*}(k,k,1;Y)+\zeta^{*}(2k,1;Y)-\zeta(k)^2Y$.
It follows that
\begin{align*}
 \sum_{n=1}^{\infty} \frac{\left(H_{n-1}^{(k)}\right)^2-\zeta(k)^2}{n}
 &= 2\zeta^{*}(k,k,1) + \zeta^{*}(2k,1).
\end{align*}
Taking $a=k-1$, this proves the first assertion for $k\geq2$.

We now consider the corner.  
It is known that
\begin{align}
 \zeta^{*}(1,1,1;Y) &= \frac16Y^3 - \frac12\zeta(2)Y + \frac13\zeta(3)
 \label{eq:zeta111-regularized}
\end{align}
and
\begin{align}
 \zeta^{*}(2,1;Y) &= \zeta(2)Y - 2\zeta(3).
 \label{eq:zeta21-regularized}
\end{align}
Setting $Y=0$ in \eqref{eq:zeta111-regularized} and \eqref{eq:zeta21-regularized}, we obtain
\begin{align*}
 2\zeta^{*}(1,1,1)+\zeta^{*}(2,1) &= -\frac43\zeta(3),
\end{align*}
which proves the second assertion.
\end{proof}

\begin{remark}\label{rem:minus-family-regularization}
The remaining case $\zeta_2\left(1-a,1/2+a+it\right)$ has the normalized coefficients $n^{-a}H_{n-1}^{(1-a)}$.
When $a=k$ is a positive integer, the exponent $1-k$ is non-positive. 
Thus this case does not directly belong to the usual harmonic regularization theory for multiple zeta values with positive integral indices.

In fact, for $k\geq2$, Faulhaber's formula with the convention $B_1=-1/2$ gives
\[
 n^{-k}H_{n-1}^{(1-k)} = \frac1k + \frac1k \sum_{j=1}^{k-1} \binom{k}{j}B_jn^{-j},
\]
where $B_j$ denotes the $j$-th Bernoulli number. 
Consequently, $C_-(k)$ can be written as a finite rational linear combination of $1$, Euler's constant $\gamma$, and ordinary zeta values.
\end{remark}


\section*{Acknowledgements}
The author would like to thank Professor Hideki Murahara for his valuable advice and helpful comments.


\section*{AI tool disclosure}
The author used OpenAI's ChatGPT during the development of this work
to discuss mathematical ideas and proof arguments, and during
manuscript preparation to improve the English exposition.
The author takes full responsibility for the mathematical claims
and the final manuscript.


\end{document}